\documentclass[11pt]{amsart}
\usepackage{amssymb,amsthm,amsmath,amstext}
\usepackage{mathdots}
\usepackage{array}
\usepackage{varwidth}
\usepackage{xcolor}
\usepackage{multirow}
\usepackage{enumitem}
\usepackage{colonequals}
\usepackage{cases}
\usepackage{graphicx}

\usepackage[left=1.0in,top=1.05in,right=1.0in,bottom=0.95in]{geometry}

\usepackage[all]{xy}

\newcommand{\g}[2]{g^{#1}_{#2}}

\DeclareMathOperator{\Pic}{Pic}

\usepackage[backref=page]{hyperref}
\hypersetup{pdftitle={Maximal Brill--Noether loci via K3 surfaces},pdfauthor={Asher Auel and Richard Haburcak}} 
\hypersetup{colorlinks=true,linkcolor=blue,anchorcolor=blue,citecolor=blue}

\usepackage[nameinlink]{cleveref}

\newtheoremstyle{thmbld}{\topsep}{\topsep}{}{}{\itshape}{}{0.5em}{}

\newtheorem{theorem}{Theorem}[section]
\newtheorem{lemma}[theorem]{Lemma}
\newtheorem{prop}[theorem]{Proposition}
\newtheorem{cor}[theorem]{Corollary}

\theoremstyle{definition}
\newtheorem{defn}[theorem]{Definition}

\theoremstyle{definition}
\newtheorem{remark}[theorem]{Remark}

\newtheorem{theoremintro}{Theorem}

\newtheorem{conjectureintro}{Conjecture}

\newtheorem{questionintro}{Question}

\newcolumntype{A}{w{c}{0.7cm}}

\title{Maximal Brill--Noether loci via K3 surfaces} 

\author{Asher Auel}
\address{Department of Mathematics\\%
	Dartmouth College\\%
	Kemeny Hall\\%
	Hanover, NH 03755}
\email{asher.auel@dartmouth.edu}

\author{Richard Haburcak}
\email{richard.haburcak.gr@dartmouth.edu}

\begin{document}

\thispagestyle{empty}

\vspace*{-.9cm}

\begin{abstract}
The Brill--Noether loci $\mathcal{M}^r_{g,d}$ parameterize curves of genus $g$ admitting a linear system of rank $r$ and degree $d$; when the Brill--Noether number is negative, they sit as proper subvarieties of the moduli space of genus $g$ curves. We explain a strategy for distinguishing Brill--Noether loci by studying the lifting of linear systems on curves in polarized K3 surfaces, which motivates a conjecture identifying the maximal Brill--Noether loci. Via an analysis of the stability of Lazarsfeld--Mukai bundles, we obtain new lifting results for line bundles of type $\g{3}{d}$ which suffice to prove the maximal Brill--Noether loci conjecture in genus $3$--$19$, $22$, and $23$.
\end{abstract}


\maketitle

\vspace*{-.4cm}

\section*{Introduction}\label{General Definitions}
Given a smooth projective complex curve $C$ of genus $g$, classical Brill--Noether theory concerns the geometry of the variety $W^r_d(C)$, parameterizing the space of line bundles of type $\g{r}{d}$, i.e., having degree $d$ and at least $r+1$ linearly independent global sections on $C$. Specifically, the expected dimension of $W^r_d(C)$ is the \textit{Brill--Noether number} $\rho(g,r,d)\colonequals g-(r+1)(g-d+r)$. In particular, when $\rho(g,r,d)\ge 0$, every smooth curve of genus $g$ admits a line bundle of type $\g{r}{d}$. If $\rho(g,r,d)<0$, then a curve admitting such a $\g{r}{d}$ is called Brill--Noether special, and the \textit{Brill--Noether locus} $\mathcal{M}^r_{g,d}$ parametrizing smooth curves of genus $g$ admitting a line bundle of type $\g{r}{d}$ is a proper subvariety of the moduli space $\mathcal{M}_g$ of smooth curves of genus $g$, see \cite{ACGH}. 

In general, the geometry of Brill--Noether loci is complicated by the existence of multiple components with some that are non-reduced or not of the expected dimension. Indeed, while the Brill--Noether locus $\mathcal{M}^r_{g,d}$ has expected codimension $-\rho$ in $\mathcal{M}_g$, the actual codimension of its components is bounded above by $-\rho$ when $\rho<0$, see e.g., \cite{Farkas_Gavril_2003}, but it could be lower, and known examples with lower than expected codimension exist when $-\rho>g-3$, see \cite{pflueger_legos}. On the other hand, when $\rho(g,r,d)=-1$, Eisenbud and Harris \cite{Eisenbud_Harris_1989} show that $\mathcal{M}^r_{g,d}$ is irreducible of codimension~$1$.  More generally, when $-3\le\rho\le-1$, any component of $\mathcal{M}^r_{g,d}$ has codimension $-\rho$, see \cite{EdidinThesis,Eisenbud_Harris_1989,Steffen1998}.  The Brill--Noether divisors were used by Harris, Mumford, and Eisenbud \cite{eisenbud_harris,harris,harris_mumford} in their investigation of the Kodaira dimension of $\mathcal{M}_g$ when $g \geq 23$.  

A question of interest is then to determine the stratification of
$\mathcal{M}_g$ by Brill--Noether loci and, in particular, to identify
those loci that are maximal with respect to containment.  For
Brill--Noether divisors, this is equivalent to having distinct
support, a property that is crucially used by Eisenbud and Harris
\cite{eisenbud_harris}, and further developed by Farkas
\cite{Farkas2000}, to give lower bounds on the Kodaira dimension of
$\mathcal{M}_{23}$.  There are various trivial containments among the
Brill--Noether loci, e.g., $\mathcal{M}^1_{g,2}\subseteq
\mathcal{M}^1_{g,3}\subseteq \cdots \subseteq
\mathcal{M}^1_{g,k}=\mathcal{M}_g$, where
$k\ge\lfloor\frac{g+3}{2}\rfloor$ is at least the generic gonality of
a curve of genus $g$. Likewise, we have
$\mathcal{M}^r_{g,d}\subseteq\mathcal{M}^{r}_{g,d+1}$ by adding a base
point to a $\g{r}{d}$ on $C$. Similarly, by subtracting a point not in
the base locus, $\mathcal{M}^r_{g,d}\subseteq
\mathcal{M}^{r-1}_{g,d-1}$ when $\rho(g,r-1,d-1)<0$, see
\cite{Farkas_2001,Lelli-Chiesa_the_gieseker_petri_divisor_g_le_13}. Modulo
these trivial containments, the \emph{expected maximal Brill--Noether
loci} are the $\mathcal{M}^r_{g,d}$, where for fixed $r$, with $2r\le
d \le g-1$, $d$ is maximal such that $\rho(g,r,d)<0$ and
$\rho(g,r-1,d-1)\ge 0$.  Hence, every Brill--Noether locus is contained
in an expected maximal one, and we conjecture that the expected
maximal loci are distinct.
%
%

\begin{conjectureintro}\label{conjecture}
In every genus $g \geq 3$, the maximal Brill--Noether loci are
the expected ones, except when $g = 7,8,9$.
\end{conjectureintro}

The conjecture states that at least one component of each expected
maximal Brill--Noether locus is not contained in any other
Brill--Noether locus, hence the expected maximal loci are indeed the
maximal elements in the containment lattice of all Brill--Noether
loci.  Concretely, this means that given any two expected maximal
Brill--Noether loci $\mathcal{M}^r_{g,d}$ and
$\mathcal{M}^{r^\prime}_{g,d^\prime}$, there exists a genus $g$ curve
admitting a $\g{r}{d}$ but not a $\g{r^\prime}{d^\prime}$.

In each genus $g=7,8,9$, there is an unexpected containment between
the two expected maximal Brill--Noether loci. In genus $8$, Mukai
\cite[Lemma 3.8]{Mukai_Curves_and_grassmannians_1993} proved the unexpected containment $\mathcal{M}^1_{8,4}
\subset \mathcal{M}^2_{8,7}$, see \Cref{genus 8 exception}. 
In genus $7$ and $9$, Hannah Larson
pointed out the unexpected containments
$\mathcal{M}^2_{7,6} \subset \mathcal{M}^1_{7,4}$ and
$\mathcal{M}^2_{9,7} \subset \mathcal{M}^1_{9,5}$, see
\Cref{genus 7 exception} and \Cref{genus 9 exception}.

Recently, there have been several breakthroughs in the study of
Brill--Noether special curves of fixed gonality
\cite{Powell_jensen_fixed_gonality,Farkas_2001,jensen2020brillnoether,Larson_Larson_Vogt_fixed_gonality,Larson_A_refined_Brill-Noether_theory_over_Hurwitz_space,pflueger2013linear,Pflueger_2017},
from which one can deduce that the expected maximal
$\mathcal{M}^1_{g,\lfloor\frac{g+1}{2}\rfloor}$ is not contained in
any of the other expected maximal loci and hence is maximal, see
\Cref{Section Maximal Brill--Noether Loci}. Additionally, following
the work of Farkas \cite{Farkas_2001} in genus $23$, there has been
recent focus on showing that Brill--Noether loci of codimension $1$
and $2$ are distinct, and showing various non-containments of
Brill--Noether loci of codimension $2$, see
\cite{CHOI2022,CHOI2012377,CHOI20141458,KIM202070}; in fact, for
$g\ge34$ and not divisible by $3$, one can deduce that there are at
least $2$ maximal Brill--Noether loci. These results are proved using
a mix of tropical, combinatorial, and limit linear series methods.

On the other hand, our approach is to use K3 surfaces to construct
curves admitting a $\g{r}{d}$, but not a $\g{r^\prime}{d^\prime}$,
thus distinguishing the Brill--Noether loci. This idea was introduced
by Farkas \cite{Farkas_2001}, and further developed by Lelli-Chiesa
\cite{Lelli_Chiesa_2013,lellichiesa2021codimension}, who can produce
curves on a K3 surface admitting a $\g{1}{d}$ or $\g{2}{d}$, but not a
$\g{r}{d^\prime}$. We further extend this technique to curves that
admit a $\g{3}{d}$, which suffices to prove our main theorem.

\begin{theoremintro}\label{Main Result 1.1}
\Cref{conjecture} holds in genus $3$--$19$, $22$, and $23$.
\end{theoremintro} 
In genus $23$, Eisenbud and Harris \cite{eisenbud_harris}, and Farkas \cite{Farkas_2001}, prove the part of this conjecture concerning the Brill--Noether divisors in their work on the birational geometry of the moduli space of curves. Concerning genus $20$ and $21$, our results reduce \Cref{conjecture} to the verification that the codimension of $\mathcal{M}^3_{20,17}$ and $\mathcal{M}^4_{21,20}$ is the expected value of~$4$, and that the codimension of $\mathcal{M}^4_{20,19}$ is at least the expected value of $5$, which should be within reach using current techniques.

The geometry of polarized K3 surfaces is intimately related to the
Brill--Noether theory of curves $C$ in the polarization class, see
e.g.,
\cite{Knutsen_k3_models_in_scrolls,Knutsen2003,Martens,MUKAI1988357,Reid1976,Saint_Donat_Proj_Models_of_K3s}. Foundational
to this is Green and Lazarsfeld's celebrated result that the Clifford
index $\gamma(C)$ is constant as $C$ moves in its linear system
\cite{GreenLaz}. Donagi and Morrison \cite[Theorem
5.1']{DonagiMorrison_linsystemsonk3} proved that if $A$ is a complete
basepoint free Brill--Noether special $\g{1}{d}$ on a
non-hyperelliptic smooth curve $C\in|H|$, then $|A|$ is contained in
the restriction of $|M|$ for a line bundle $M\in\Pic(S)$.  In fact,
they conjectured that this is always true, with some slight
modifications due to Lelli-Chiesa.

\begin{conjectureintro}[Donagi--Morrison Conjecture, \cite{Lelli_Chiesa_2015} Conjecture 1.3]\label{conj DM}
	Let $(S,H)$ be a polarized K3 surface and $C\in|H|$ be a smooth irreducible curve of genus $\ge 2$. Suppose $A$ is a complete basepoint free $\g{r}{d}$ on $C$ such that $d\le g-1$ and $\rho(g,r,d)<0$. Then there exists a line bundle $M\in\Pic(S)$ adapted to $|H|$ such that $|A|$ is contained in the restriction of $|M|$ to $C$ and $\gamma(M\otimes\mathcal{O}_C)\le \gamma(A)$.
\end{conjectureintro}

For further details and definitions, such as the notion of
\emph{adapted}, 
see \Cref{subsection Philosophy}.  Lelli-Chiesa has verified the
Donagi--Morrison conjecture for linear systems of type $\g{2}{d}$
under some mild hypotheses \cite{Lelli_Chiesa_2013}, and more recently
\cite{Lelli_Chiesa_2015}, has proven the conjecture if the pair
$(C,A)$ does not have unexpected secant varieties up to
deformation. Importantly, in \cite[Appendix~A]{Lelli_Chiesa_2015},
Lelli-Chiesa and Knutsen construct explicit examples that show
\Cref{conj DM} is in general false for linear systems rank $3$, see
\Cref{remark:counterexample}. The proofs of
these results use Lazarsfeld--Mukai bundles $E_{C,A}$ associated to
the pair $(C,A)$, and the fact that when the vector bundle $E_{C,A}$
has a nontrivial maximal destabalizing sub-line bundle $N\in\Pic(S)$,
then $|A|$ is contained in the restriction of $|H\otimes
N^{\vee}|$. For rank $2$ linear systems, a case-by-case analysis of
the Jordan--H\"{o}lder and Harder--Narasimhan filtrations of $E_{C,A}$
is used. This technique becomes much more difficult in higher rank.
In general, Lelli-Chiesa \cite[Theorem 4.2]{Lelli_Chiesa_2015} proves
that $A$ does lift when it computes the Clifford index $\gamma(C)$.
However, in genus $g \geq 14$, except for
$\mathcal{M}^1_{g,\lfloor\frac{g+1}{2}\rfloor}$, all of the expected
maximal Brill--Noether loci correspond to \emph{non-computing}
Brill--Noether special linear systems, i.e., linear systems $|A|$ with
$\rho(A)<0$ and $\gamma(A)>\lfloor\frac{g-1}{2}\rfloor$ so that $A$
cannot compute the Clifford index, as
$\gamma(C)\le\lfloor\frac{g-1}{2}\rfloor$.

Since one cannot hope to prove \Cref{conj DM} in general, our main lifting result is a proof of the Donagi--Morrison conjecture for linear systems of rank~$3$ and bounded degree.

\begin{theoremintro}\label{Main Result 1.2}
	Let $(S,H)$ be a polarized K3 surface of genus $g\neq 2,3,4,8$ and $C\in|H|$ a smooth irreducible curve of Clifford index $\gamma(C)$. Suppose that $S$ has no elliptic curves and $d<\frac{5}{4}\gamma(C)+6$, then \Cref{conj DM} holds for any $\g{3}{d}$ on $C$. Moreover, one has $c_1(M).C\le\frac{3g-3}{2}$.
\end{theoremintro}
We prove a slightly more refined version, replacing the hypothesis on non-existence of elliptic curves with an explicit dependence on the Picard lattice of $S$, see \Cref{theorem lifting g3ds general}.

With this lifting result in hand, \Cref{Main Result 1.1} is proved by considering K3 surfaces $(S,H)$ with a prescribed Picard group so that curves $C\in|H|$ have a $\g{r}{d}$, and then proving that if $C$ had a $\g{3}{d^\prime}$, its Donagi--Morrison lift would not be compatible with the Picard group. This latter argument involves some elementary lattice theory. More generally, we explain how a Donagi--Morrison type result together with some lattice theory imply \Cref{conjecture}. As the Donagi--Morrison conjecture is not known in rank $4$ and above, we cannot show that some of the expected maximal Brill--Noether loci are not contained in the $\mathcal{M}^4_{g,d}$ in genus $20$, $21$, and $\ge 24$. In genus $22$ and $23$, known results about the codimension of components of Brill--Noether loci and non-containments of codimension $2$ loci, together with our results, suffice to distinguish the expected maximal loci.

\subsection*{Outline} In \Cref{Section Maximal Brill--Noether Loci}, we briefly analyze some constraints on lifting line bundles and find that in genus $\ge14$ the expected maximal Brill--Noether loci correspond to line bundles that cannot compute the Clifford index of the curve, and summarize how \Cref{conj DM} implies \Cref{conjecture}. The following two sections, \Cref{Stability of Sheaves on K3 Surfaces} and \Cref{Section Lazardsfeld--Mukai Bundles and Lifting and Generalized LM Bundles}, provide some background on the notion of stability of coherent sheaves on K3 surfaces and on Lazarsfeld--Mukai bundles and their relation to lifting line bundles. We also briefly recall some useful facts about generalized Lazarsfeld--Mukai bundles which are needed in particular arguments. At the end of \Cref{Section Lazardsfeld--Mukai Bundles and Lifting and Generalized LM Bundles}, we motivate our proof strategy in \Cref{Prop Proof Strategy}.  In \Cref{Section Filtrations of Lazarsfeld--Mukai Bundles of Rank 4}, we first reduce the problem to finding a bound for each terminal filtration of the Lazarsfeld--Mukai bundle associated to the $\g{3}{d}$, a filtration obtained by taking the Harder--Narasimhan and Jordan--H\"{o}lder filtrations of the Lazarsfeld--Mukai bundle. We then find a bound on the degree of the $\g{3}{d}$ for each filtration. In \Cref{Section Lifting \texorpdfstring{$\g{3}{d}$}{}s}, after having obtained bounds for every terminal filtration that does not have a maximal destabilizing sub-line bundle, we give the proof of \Cref{Main Result 1.2}. Finally, in \Cref{Section Maximal Brill--Noether Loci in Genus \texorpdfstring{$14-23$}{}}, we use known results about dimensions of components of Brill--Noether loci and other lifting results to prove \Cref{Main Result 1.1}. In \Cref{Section Lower Genus}, we prove the results in genus $3$--$13$, and the following sections summarize genus $14$--$23$. 

\subsection*{Acknowledgments} 
The authors would like to thank Gabi Farkas for elucidating
connections between Brill--Noether loci and the birational geometry of
$\mathcal{M}_g$, David Jensen and Nathan Pflueger for very helpful
conversations about Brill--Noether theory, Hannah Larson for pointing out the exceptional case of genus $9$, Margherita Lelli-Chiesa for
explaining technical aspects of her work, Andreas Leopold Knutsen for
helpful comments on the Donagi--Morrison conjecture, Isabel Vogt for comments on a draft of the paper and pointing out recent relevant literature, and John Voight for helpful
computations and comments on a draft of the paper.  The first author
is partially supported by Simons Foundation Collaboration Grant
712097 and National Science Foundation Grant DMS-2200845.

\section{Maximal Brill--Noether Loci}\label{Section Maximal Brill--Noether Loci}
In this section, we take a look at the analytic geometry of various Brill--Noether theory conditions on linear systems. We find simple bounds on the maximal Clifford index of Brill--Noether special linear systems and for linear systems that can potentially lift to a K3 surface without contradicting the Hodge index theorem. Furthermore, we find that all non-computing linear systems are always potentially liftable to K3 surfaces. We end with a discussion of how \Cref{conj DM} and lattice theory can imply \Cref{conjecture}. We work with a fixed genus $g$ throughout this section.

Let $(S,H)$ be a polarized K3 surface of genus $g$. In the moduli space $\mathcal{K}^\circ_g$ of polarized K3 surfaces of genus $g$, the Noether--Lefschetz (NL) locus parameterizes K3 surfaces with Picard rank $>1$. By Hodge theory, the NL locus is a union of countably many irreducible divisors, which we call NL divisors. In \cite{greer-li-tian2014}, Greer, Li, and Tian study the Picard group of $\mathcal{K}^\circ_g$ using Noether--Lefschetz theory and the locus of Brill--Noether special K3 surfaces in $\mathcal{K}^\circ_g$ is identified as a union of NL divisors. More generally, it is convenient to work with the moduli space of primitively quasi-polarized K3 surfaces, denoted $\mathcal{K}_g$ where $\mathcal{K}_g\setminus \mathcal{K}^\circ_g$ is a divisor parameterizing K3 surfaces containing a $(-2)$-exceptional curve. We define the NL divisor $\mathcal{K}^r_{g,d}$ to be the locus of polarized K3 surfaces $(S,H)\in\mathcal{K}_g$ such that \[\Lambda^r_{g,d}=\begin{array}{c|cc}
	\multicolumn{1}{c}{} & H & L \\\cline{2-3}
	H & 2g-2 &d \\
	L &d & 2r-2
\end{array}\] admits a primitive embedding in $\Pic(S)$ preserving $H$. We note that the $\mathcal{K}^r_{g,d}$ are each irreducible by \cite{OGrady_irreducible_NL_divisors}. As we'll show in \Cref{lemma max gon and Cliff dim 1}, polarized K3 surfaces $(S,H)\in\mathcal{K}^r_{g,d}$ should be thought of as those having a curve $C\in|H|$ such that $L\otimes\mathcal{O}_C$ is a line bundle of type $\g{r}{d}$, and we say that the lattice $\Lambda^r_{g,d}$ is \emph{associated} to $\g{r}{d}$. Specifically, we have the following lemma, which we prove in \Cref{Section Maximal Brill--Noether Loci in Genus \texorpdfstring{$14-23$}{}}.

\begin{lemma}[See \Cref{lemma max gon and Cliff dim 1}]
	Let $(S,H)\in\mathcal{K}^r_{g,d}$ and let $C\in |H|$ be a smooth irreducible curve. If $L$ and $H-L$ are basepoint free, $r\ge 2$, and $1\le d \le g-1$, then $L\otimes\mathcal{O}_C$ is a $\g{r}{d}$.
\end{lemma}

Conversely, one is interested in the question of when a given $\g{r}{d}$ on a curve in a K3 surface is the restriction of a line bundle from the K3; in this case, we say that the line bundle is a \emph{lift} of the $\g{r}{d}$. Lifting of line bundles on curves on K3 surfaces is considered in \cite{DonagiMorrison_linsystemsonk3,GreenLaz,Lelli_Chiesa_2013,Lelli_Chiesa_2015,Martens,Reid1976}. In lifting Brill--Noether special linear systems on $C\in|H|$ to a line bundle $L\in\Pic(S)$, we are naturally led to considering two constraints. First, we have $\rho(g,r,d)<0$ as the linear system is Brill--Noether special. We call the constraint $\rho(g,r,d)<0$ the \emph{Brill--Noether constraint}. If a $\g{r}{d}$ on a curve $C\in|H|$ on a polarized K3 surface $(S,H)$ has a suitable lift (see \Cref{LC implies 2r-2}), then $\Pic(S)$ admits a primitive embedding of $\Lambda^r_{g,d}$ preserving $H$, and in particular $\operatorname{disc}\bigl(\Lambda^r_{g,d}\bigr)<0$ by the Hodge index theorem. Thus we define \[\Delta(g,r,d)\colonequals\operatorname{disc}\left(\Lambda^r_{g,d}\right)=4(g-1)(r-1)-d^2=4(g-1)(r-1)-(\gamma(r,d)+2r)^2.\]  We thus call the constraint $\Delta(g,r,d)<0$ the \emph{Hodge constraint} as the inequality stems from the Hodge index theorem. We remark that when $\Delta(g,r,d)<0$, the Torelli theorem for polarized K3 surfaces implies that a very general K3 surface in $\mathcal{K}^r_{g,d}$ has $\Pic(S)=\Lambda^r_{g,d}$.

\begin{remark}\label{remark bounds on BN and Hodge par}
	When considering the lifting of linear systems to K3 surfaces, it is more convenient to consider the Brill--Noether and Hodge constraints for fixed $g$ in the $(r,\gamma)$-plane as opposed to the $(r,d)$-plane, in particular, because the Clifford index of curves on K3 surfaces remains constant in their linear system \cite{GreenLaz}. In the $(r,\gamma)$-plane the Brill--Noether and Hodge constraints determine regions that are bounded by the curves $\rho(g,r,d)=0$ and $\Delta(g,r,d)=0$, which we call the \emph{Brill--Noether hyperbola} and  \emph{Hodge parabola}, respectively. Simple calculations show that the maximum $\gamma$ on the Brill--Noether hyperbola is obtained at $r=\sqrt{g}-1$ and $\gamma=g-2\sqrt{g}+1$, the intersection with the line $d=g-1$. Hence, taking $\gamma\le \lfloor g-2\sqrt{g}+1\rfloor$ suffices to bound Brill--Noether special linear systems. Similarly, the maximum $\gamma$ on the Hodge parabola is given by $\gamma=\frac{g-5}{2}$, and obtained at the intersection with the line $d=g-1$ at $r=\frac{g+3}{4}$. Thus if $\gamma>\frac{g-5}{2}$ then $\Delta<0$. Trivially $\lfloor \frac{g-4}{2}\rfloor\ge \frac{g-5}{2}$, and in fact the bound $\gamma\ge \lfloor\frac{g-4}{2}\rfloor \implies \Delta<0$ is the best possible as seen in genus 9, 13, and 17. As an example, we show the bounds in genus $100$, as graphed on the $(r,\gamma)$-plane in \Cref{fig:bnhodgeparabolasgenus100}.
\end{remark}
\begin{figure}
	\centering
	\includegraphics[width=0.7\linewidth]{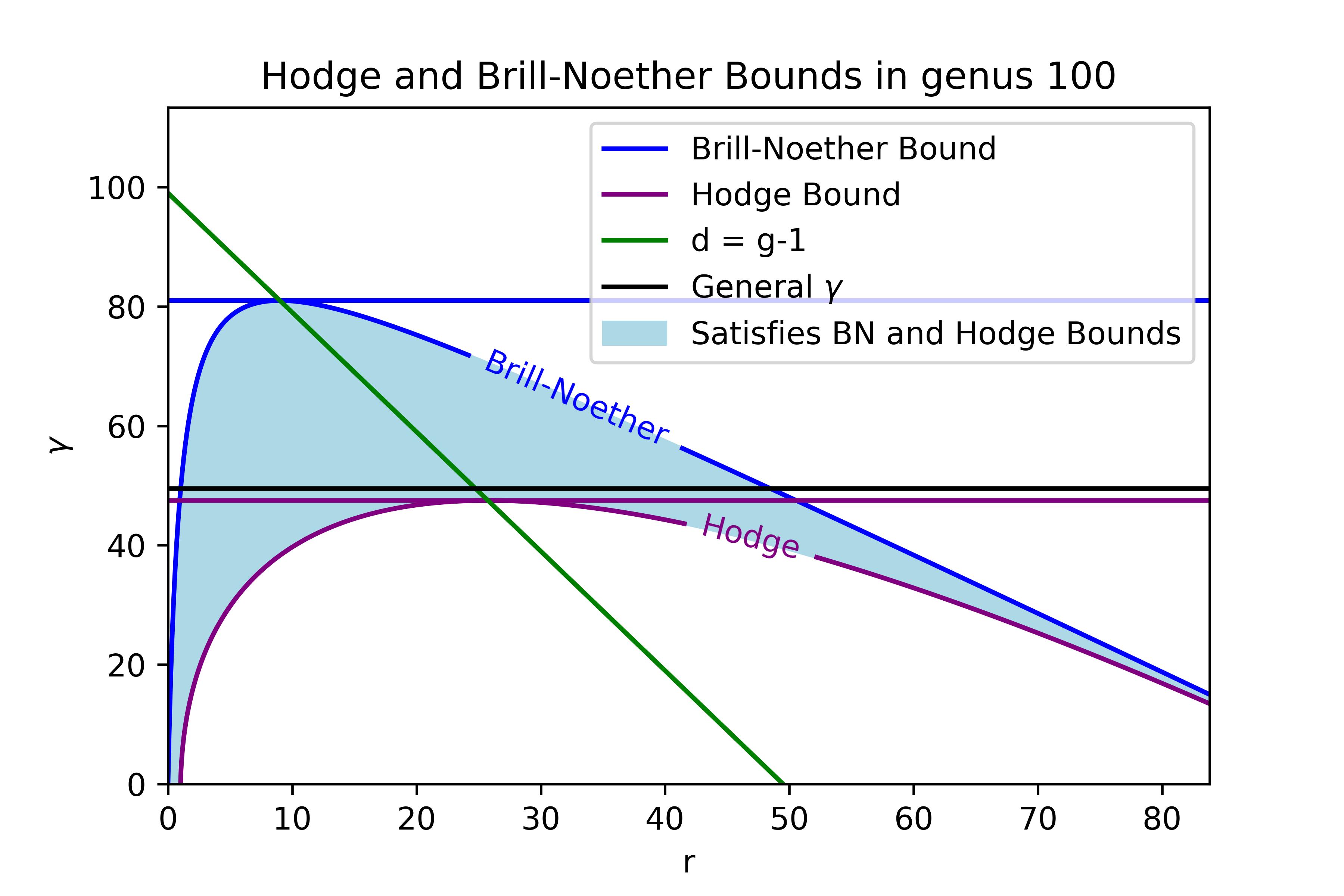}\\
	\caption[]{The Brill--Noether hyperbola ($\rho=0$) and the Hodge parabola ($\Delta=0$) in genus $100$. The shaded area satisfies both $\rho<0$ and $\Delta<0$.}
	
	\label{fig:bnhodgeparabolasgenus100}
\end{figure}

We recall that the \emph{Clifford index} of a line bundle $A$ on a smooth
projective curve $C$ is the integer $\gamma(A) = \deg(A) - 2\, r(A)$
where $r(A) = h^0(C,A)-1$ is the rank of $A$. The Clifford
index of $C$ is $$\gamma(C) \colonequals \min \{\gamma(A) ~\vert~ h^0(C,A) \geq 2\text{ and }h^1(C,A) \geq 2 \}.$$ We say that a line bundle $A$ on $C$ \emph{computes} the Clifford index of $C$ if
$\gamma(A)=\gamma(C)$. Clifford's theorem states that $0\le\gamma(C)\le\lfloor\frac{g-1}{2}\rfloor$, and when $C$ is a general curve of genus $g$, $\gamma(C)=\lfloor\frac{g-1}{2}\rfloor$.
\begin{defn}
Let $A$ be a Brill--Noether special $\g{r}{d}$ on a curve $C$ of genus
$g$, i.e. $\rho(g,r,d)<0$. We say $A$ is \emph{non-computing} if
$\gamma(r,d)>\lfloor \frac{g-1}{2} \rfloor$, that is, $A$ is a
Brill--Noether special $\g{r}{d}$ that cannot compute the Clifford
index of $C$.
\end{defn}
\begin{lemma}\label{lemma exp max loci have large gamma}
	Let $g\ge 14$, $r\ge 2$, and $2r\le d \le g-1$. If $\mathcal{M}^r_{g,d}$ is an expected maximal Brill--Noether locus, then $\gamma(d,r) = d - 2r > \lfloor\frac{g-1}{2}\rfloor$. When $g < 14$, there are no non-computing $\g{r}{d}$'s.
\end{lemma}
\begin{proof}
	One can easily check that if $d-2r\le\lfloor\frac{g-1}{2}\rfloor$, then  $\rho(g,r,d+1)<0$, and hence $\mathcal{M}^r_{g,d}$ is not an expected maximal Brill--Noether locus. When $g<14$, this is a simple computation enumerating all $\g{r}{d}$'s with Clifford index $\le \lfloor\frac{g-1}{2}\rfloor+1$. 
\end{proof}
Thus for genus $g\ge 14$, except for $\mathcal{M}^1_{g,\lfloor\frac{g+1}{2}\rfloor}$, all the maximal Brill--Noether loci are those associated to non-computing $\g{r}{d}$s. If lifting results are able to distinguish between maximal Brill--Noether loci, there should not be an obvious obstruction to lifting the associated linear systems. In particular, the Hodge index theorem implies that the lattices obtained by lifting should have negative discriminant, which we show is true for non-computing $\g{r}{d}$s below.

\begin{prop}
	Let $g,r,d$ be natural numbers with $2\le d\le g-1$ and $1\le r\le g-1$. Then the Hodge parabola lies under the Brill--Noether hyperbola. In particular, all non-computing linear systems, and all expected maximal Brill--Noether loci, satisfy $\Delta<0$.
\end{prop}
\begin{proof}
	For fixed $g\ge 2$, and for each constraint ($\rho=0$ or $\Delta=0$), we solve for $\gamma$ as a function of $r$ and $g$. For $\rho(g,r,\gamma)=0$, we find $\gamma_{\rho}(r)=g-r-\frac{g}{r+1}$. Likewise for $\Delta(g,r,\gamma)=0$ we have $\gamma_{\Delta}(r)=2\sqrt{(g-1)(r-1)}-2r$. Observe that $\gamma_{\rho}=\gamma_{\Delta}$ has no solutions in the given range (solve for $r$ in terms of $g$, and note that $g\ge 2$). Finally, since $\gamma_{\rho}(1)>0$ and $\gamma_{\Delta}(1)<0$, we see by continuity that $\gamma_{\rho}(r)-\gamma_{\Delta}(r)>0$. 
	
	The bound $\gamma\ge\lfloor \frac{g-4}{2}\rfloor$ implies that $\Delta<0$, as in the remark above. Since this is below the general Clifford index ($\lfloor\frac{g-1}{2}\rfloor)$, we see that any lattice associated to a non-computing linear system will have negative discriminant. In particular, by \Cref{lemma exp max loci have large gamma} above, this applies to the expected maximal linear systems.
\end{proof}

We thus conjecture (\Cref{conjecture}) that the maximal Brill--Noether loci are exactly the \emph{expected maximal Brill--Noether loci}, which are Brill--Noether loci $\mathcal{M}^r_{g,d}$ where for fixed $r$, $d$ is maximal such that $\rho(g,r,d)<0$ and $\rho(g,r-1,d-1)\ge 0$. Equivalently, the expected maximal Brill--Noether loci correspond to the maximal $\g{r}{d}$ lying under the Brill--Noether hyperbola for each $r$, up to the containments $\mathcal{M}^r_{g,d}\subseteq \mathcal{M}^r_{g,d+1}$ when $\rho(g,r,d+1)<0$ and $\mathcal{M}^r_{g,d}\subset\mathcal{M}^{r-1}_{g,d-1}$ when $\rho(g,r-1,d-1)<0$.

\medskip

One could imagine that if there are any unexpected containments among Brill--Noether loci, then some would come from containments of the form $\mathcal{M}^1_{g,d}\subset \mathcal{M}^r_{g,d^\prime}$. However, we find that the expected maximal $\mathcal{M}^1_{g,d}$ is not contained in the other expected maximal loci.

\begin{prop}\label{Prop submax gonality not contained in noncomputing locus}
	Let $\rho(g,r,d)<0$, and $\gamma(r,d)\ge \lfloor\frac{g-1}{2}\rfloor+1$, e.g., for a non-computing $\g{r}{d}$. Then $\mathcal{M}^1_{g,\lfloor\frac{g+1}{2}\rfloor}\nsubseteq\mathcal{M}^r_{g,d}$. When $9\le g<14$, $\mathcal{M}^1_{g,\lfloor\frac{g+1}{2}\rfloor}\nsubseteq \mathcal{M}^r_{g,d}$ for an expected maximal Brill--Noether locus with $r\ge 2$.
\end{prop}
\begin{proof}
	Let $k=\frac{g+1}{2}$, and $r^\prime=\min\{r,g-d+r-1\}$. We compute 
	\begin{align*}
		\rho_k&=\max_{\ell\in\{0,\dots,r^\prime\}} \rho(g,r-\ell,d)-\ell k = \rho(g,r,d)+(g-k-\gamma(r,d)+1)\ell - \ell^2\\ &\le\max_{\ell\in\{0,\dots,r^\prime\}}\rho(g,r,d)+\ell\left(g-\left\lfloor\frac{g-1}{2} \right\rfloor - \left\lfloor \frac{g+1}{2}\right\rfloor\right)-\ell^2\\
		&<\max_{\ell\in\{0,\dots,r^\prime\}}\rho(g,r,d)+2\ell-\ell^2 \le \rho(g,r,d)+1\le0.
	\end{align*}
	Therefore $\rho_k<0$. From \cite[Theorem 1.1]{Pflueger_2017}, as $\dim W^r_d(C)\le\rho_k$, and $W^r_d(C)$ is empty if its dimension is negative, we see that a general $k$-gonal curve does not admit a $\g{r}{d}$. Hence $\mathcal{M}^1_{g,\lfloor\frac{g+1}{2}\rfloor}\nsubseteq\mathcal{M}^r_{g,d}$. 
	
	The statement for $9\le g<14$ is obtained simply by calculating $\rho_k$ explicitly, and noting that in each case $\rho_k<0$. 
\end{proof}
\begin{remark}
	In \cite{Farkas_2001}, Farkas asks the general question of when does a general $k$-gonal curve of genus $g$ have no other linear series $\g{r}{d}$ with $\rho(g,r,d)<0$? The above proposition answers the case when $k=\lfloor\frac{g+1}{2}\rfloor$, when the curve has maximal sub-general gonality. If a curve has a Brill--Noether special $\g{r^\prime}{d^\prime}$, then it has a $\g{r}{d}$ for an expected maximal Brill--Noether locus, and the above shows this is not the case. In general, this question is answered by recent breakthroughs in Brill--Noether theory for curves of fixed gonality, see e.g. \cite{Powell_jensen_fixed_gonality,Farkas_2001,jensen2020brillnoether,Larson_Larson_Vogt_fixed_gonality,Larson_A_refined_Brill-Noether_theory_over_Hurwitz_space,pflueger2013linear,Pflueger_2017}. 

	In \Cref{lemma max gon and Cliff dim 1}, we show that under mild assumptions, the curves $C\in|H|$ on a polarized K3 surface $(S,H)$ with $\Pic(S)=\Lambda^r_{g,d}$ associated to an expected maximal locus with $r\ge 2$, all have general Clifford index. Thus the $\mathcal{M}^1_{g,\lfloor\frac{g+1}{2}\rfloor}$ does not contain other expected maximal loci in many genera. Similar results have been proven by Fakras and Lelli-Chiesa \cite{Farkas_2001,Lelli_Chiesa_2013}. 
\end{remark}

A natural question is whether lattices corresponding to $\g{r}{d}$s can be contained as sublattices in each other. In general, the answer is yes. Already in genus $14$, we see that $\Lambda^2_{14,10}$ could be embedded as a sublattice of $\Lambda^2_{14,8}$. However, these are not associated to expected maximal loci. In particular, we would like to show that lattices associated to expected maximal loci cannot contain any lattices associated to other $\g{r}{d}$. This turns out to be false (see \Cref{subsection Philosophy}). However, we can prove that lattices associated to Brill--Noether special linear systems with lower than general Clifford index cannot be contain in lattices associated to expected maximal loci, and that any containments between lattices associated to an expected maximal loci and those associated to non-computing $\g{r}{d}$s must be equalities.

\begin{prop}\label{prop lattices are not conatined in exp max} Let $\Lambda^r_{g,d}$ be associated to an expected maximal $\g{r}{d}$.
	\begin{enumerate}[label=(\roman*)]
		\item Any lattice $\Lambda^{r^\prime}_{g,d^\prime}$ associated to a special $\g{r^\prime}{d^\prime}$ with $\gamma(\g{r^\prime}{d^\prime})<\lfloor\frac{g-1}{2}\rfloor$ for any $r^\prime$ or $\gamma(\g{r^\prime}{d^\prime})=\lfloor\frac{g-1}{2}\rfloor$ if $r^\prime \neq 1$  cannot be contained in $\Lambda^r_{g,d}$.
		
		\item Let $d^\prime \le g-1$. Any lattice $\Lambda^{r^\prime}_{g,d^\prime}$ associated to another expected maximal $\g{r^\prime}{d^\prime}$ is not contained in $\Lambda^r_{g,d}$, unless the lattices are isomorphic. Similarly, any lattice associated to a non-computing $\g{r^\prime}{d^\prime}$ with $d^\prime\le g-1$ is not contained in the lattice associated to an expected maximal $\g{r}{d}$ unless they are isomorphic.
	\end{enumerate}
\end{prop}
\begin{proof}
	To simplify notation, we write $\Delta$ for the discriminant of a lattice $\Lambda$.

	To prove $(i)$, we recall that if $\Lambda_{sub} \subset \Lambda_{exp}$ is a finite index sublattice, then we have $\Delta_{sub}=[\Lambda_{exp}:\Lambda_{sub}]^2\Delta_{exp}$. We calculate that the ratio $\frac{\Delta_{sub}}{\Delta_{exp}}$ is never a square for the lattices considered. Specifically, we show that the largest negative discriminant $-\Delta_{sub}$ among lattices with $\gamma<\lfloor\frac{g-1}{2}\rfloor$, divided by the negative discriminant $-\Delta_{exp}$ of any lattice associated to an expected maximal linear system, is not an integer. 
	Because $\Delta(g,r,d)=\operatorname{disc}(-\Lambda^r_{g,d})=d^2-4(g-1)(r-1)$, it is clear that for fixed $\gamma$ this decreases as $r$ increases until $d=g-1$. It follows that none of the lattices considered can be contained in $\Lambda^1_{g,\lfloor\frac{g+1}{2}\rfloor}$, the expected maximal loci with $r=1$. From now on, we assume $r>1$. Furthermore, we can take
	\begin{itemize}
		\item $\max(-\Delta_{sub})=d^2$ with $d=\frac{g+1}{2}$ when $\gamma=\frac{g-1}{2}-1$ ; or
		\item $\max(-\Delta_{sub})=d^2-4(g-1)$ with $d=\frac{2}{3}g+2$ when $\gamma=\frac{g-1}{2}$.
	\end{itemize}

	We also note that $-\Delta$ increases when $r$ and $\gamma$ both increase by $1$, and increases as $\gamma$ increases for fixed $r$. Thus if $r^\prime \ge r$, then clearly $\frac{\max(-\Delta_{sub})}{-\Delta_{exp}}<1$. If $r^\prime<r$, then moving from $\g{r^\prime}{d^\prime}$ to $\g{r}{d}$, we take steps increasing $r^\prime$ and $\gamma$ by 1 until we hit $r$ (and then take steps increasing $\gamma$) or hit the line $d=g-1$ and we take steps increasing $\gamma$ by $1$ and decreasing $r^\prime$ by $1$. Since each of these steps increase $-\Delta$, we again see that $\frac{\max(-\Delta_{sub})}{-\Delta_{exp}}<1$. We can always take these steps since we may assume we start at $r=1$ or $r=2$, and the expected maximal $\g{r}{d}$ lie far above. Thus $(i)$ is proved. 
	
	To prove $(ii)$, we similarly bound $\max(-\Delta)$ and $\min(-\Delta)$ for non-computing $\g{r}{d}$s. It can be verified that the ratio $\frac{\min(-\Delta)}{\max(-\Delta)}>\frac{1}{4}$ for $r<\sqrt{g}$, and hence $\max(-\Delta)<4\min(-\Delta)$, thus the discriminants of lattices associated to the expected maximal Brill--Noether loci cannot differ by a square greater than $1$. Hence if the lattices associated to expected maximal loci are contained, they must be the same lattice. Since $-\Delta$ increases as $r$ decreases and as $\gamma$ increases until $d=g-1$, this argument in fact shows that any lattice associated to a non-expected maximal non-computing $\g{r^\prime}{d^\prime}$ cannot be contained in the lattice of an expected maximal $\g{r}{d}$ unless they have the same discriminant.
\end{proof}
\begin{remark}
	In fact, computation up to large genus shows that the lattices associated to expected maximal loci do not contain any lattices associated to other expected maximal loci. We conjecture that this is always true, though a proof of this is currently unknown.
\end{remark}

\subsection{Program: Donagi--Morrison implies maximal Brill--Noether loci}\label{subsection Philosophy}
To verify \Cref{conjecture}, our strategy is for fixed genus $g$ and distinct expected maximal $\mathcal{M}^r_{g,d}$ and $\mathcal{M}^{r^\prime}_{g,d^\prime}$ to prove that for a very general K3 surface $(S,H)\in\mathcal{K}^r_{g,d}$, a smooth curve $C\in|H|$ admits a $\g{r}{d}$ but not a $\g{r^\prime}{d^\prime}$. We do this by combining three kinds of results: (i) a lifting result, (ii) showing that $C\in|H|$ has a $\g{r}{d}$ given by restricting $L\in\Lambda^r_{g,d}$, and (iii) a comparison result that distinguishes lattices. The latter two can be checked for any fixed genus. If all the lattices can be distinguished, a lifting result like the Donagi--Morrison conjecture (\Cref{conj DM}) implies \Cref{conjecture}.

We start by defining a few terms in \Cref{conj DM}.

\begin{defn}\label{definition linear system contained in restriction of linear system}
	Let $S$ be a K3 surface, $C\subset S$ be a curve, and $A\in\Pic(C)$ and $M\in\Pic(S)$ be line bundles. We say that the linear system $|A|$ is contained in the restriction of $|M|$ to $C$ when for every $D_0\in|A|$, there is some divisor $M_0\in|M|$ such that $D_0\subset C\cap M_0$.
\end{defn}	
\begin{defn}\label{definition adapted}
	A line bundle $M$ is \emph{adapted} to $|H|$ when \begin{enumerate}[label=(\roman*)]
			\item $h^0(S,M)\ge 2$ and $h^0(S,H\otimes M^\vee)\ge2$; and
			\item $h^0(S,M\otimes\mathcal{O}_C)$ is independent of the smooth curve $C\in|H|$.
		\end{enumerate}  
\end{defn}
Thus whenever $M$ is adapted to $|H|$, condition (i) ensures that $M\otimes\mathcal{O}_C$ contributes to $\gamma(C)$, and condition (ii) ensures that $\gamma(M\otimes\mathcal{O}_C)$ is constant as $C$ varies in its linear system and is satisfied if either $h^1(S,M)=0$ or $h^1(S,H\otimes M^\vee)=0$.
\begin{defn}
\label{defn:DM_lift}
	Let $(S,H)$ be a polarized K3 surface and $C\in|H|$ be a smooth irreducible curve of genus $\ge 2$. Suppose $A$ is a complete basepoint free $\g{r}{d}$ on $C$ such that $d\le g-1$ and $\rho(g,r,d)<0$. We call a line bundle $M$ a \emph{Donagi--Morrison lift} of $A$ if $M$ satisfies the conditions in \Cref{conj DM}. That is,
	\begin{itemize}
		\item $M$ is adapted to $|H|$,
		\item $|A|$ is contained in the restriction of $|M|$ to $C$, and
		\item $\gamma(M\otimes\mathcal{O}_C)\le \gamma(A)$.
	\end{itemize}
	We call a line bundle $M$ a \emph{potential Donagi--Morrison lift} of $A$ if $M$ satisfies $\gamma(M\otimes\mathcal{O}_C)\le\gamma(A)$ and $d(M\otimes\mathcal{O}_C)\ge d(A)$. Note that a Donagi--Morrison lift is a potential Donagi--Morrison lift. We say a (potential) Donagi--Morrison lift is of type $\g{s}{e}$ if $M^2=2s-2$ and $M.H=e$. 
\end{defn}

We summarize a few potential results distinguishing lattices, each of which would be useful in verifying \Cref{conjecture} given an appropriate lifting result.
\begin{itemize}
\label{Expected lattices dont contain sub lattices}
\renewcommand{\itemsep}{1mm}
	\item[\textbf{(L1)}] For a fixed lattice $\Lambda^r_{g,d}$ associated to an expected maximal $\mathcal{M}^{r}_{g,d}$ and any lattice $\Lambda^{r^\prime}_{g,d^\prime}$ associated to another expected maximal $\mathcal{M}^{r^\prime}_{g,d^\prime}$, one has $\Lambda^{r^\prime}_{g,d^\prime}\nsubseteq\Lambda^r_{g,d}$.
	\item[\textbf{(L2)}] For a fixed lattice $\Lambda^r_{g,d}$ associated to an expected maximal $\mathcal{M}^{r}_{g,d}$ and any lattice $\Lambda^{r^\prime}_{g,d^\prime}$ with $\lfloor\frac{g+1}{2}\rfloor \le \gamma(r^\prime,d^\prime) \le \lfloor g-2\sqrt{g}+1\rfloor$ and $1\le r^\prime\le\lfloor\frac{g-1-\gamma(r^\prime,d^\prime)}{2}\rfloor$, one has $\Lambda^{r^\prime}_{g,d^\prime}\nsubseteq\Lambda^r_{g,d}$.
	\item[\textbf{(L3)}] For a pair of lattices $(\Lambda^r_{g,d}, \Lambda^{r^\prime}_{g,d^\prime})$ both associated to expected maximal Brill--Noether loci, and any lattice $\Lambda^s_{g,e}$ such that $\lfloor\frac{g+1}{2}\rfloor \le \gamma(s,e) \le \gamma(r^\prime,d^\prime)$ and $1\le s\le\lfloor\frac{g-1-\gamma(s,e)}{2}\rfloor$, one has $\Lambda^{s}_{g,e}\nsubseteq\Lambda^r_{g,d}$.
\end{itemize}
We note that L2 implies L1. Furthermore, for fixed $r$ and $d$, L2 implies L3 for all $r^\prime$ and $d^\prime$.

\begin{remark}
	The bounds on $\gamma(s,e)$ and $s$ in L3 include all lattices associated to a potential Donagi--Morrison lift of a $\g{r^\prime}{d^\prime}$. Indeed, suppose $M$ is a potential Donagi--Morrison lift of a $\g{r^\prime}{d^\prime}$, and say $M$ is of type $\g{s}{e}$. The lower bound on $\gamma(s,e)$ comes from \Cref{prop lattices are not conatined in exp max} $(i)$. Since $M$ is a potential Donagi--Morrison lift of a $\g{r^\prime}{d^\prime}$, we have $\gamma(s,e)\le\gamma(r^\prime,d^\prime)$, which is the upper bound on $\gamma(s,e)$. Since $M\otimes\mathcal{O}_C$ contributes to $\gamma(C)$, this forces $H\otimes M^\vee\otimes\mathcal{O}_C$ to be at least a $g^1_{2g-2-e}$, whereby $s \le \frac{g-1-\gamma(s,e)}{2}$ as $2s\le e$, which gives the upper bound on $s$.
	
	Similarly, the bounds in L2 include all lattices associated to a potential Donagi--Morrison lift of an expected maximal linear system. $M\otimes\mathcal{O}_C$ must have Clifford index no bigger than the expected maximal $\g{r}{d}$ by \Cref{conj DM}, the upper bound on $\gamma(r^\prime,d^\prime)$ comes from \Cref{remark bounds on BN and Hodge par}. The other bounds are obtained in the same way as for L3. 	 	
\end{remark}

\begin{remark}
	As stated above, computations show that L1 holds for every expected maximal locus up to large genus.
		
	We note that L2 and L3 do not always hold. The first genus where L3 fails is $g=56$, where L3 fails for the lattices $\Lambda^r_{g,d}=\Lambda^2_{56,39}$ and $\Lambda^{r^\prime}_{g,d^\prime}=\Lambda^6_{56,49}$; indeed, in attempting to check whether $\mathcal{M}^2_{56,39}$ can be contained in $\mathcal{M}^3_{56,44}$, a $\g{3}{44}$ on a curve $C\in|H|$ for a very general $(S,H)\in\mathcal{K}^2_{56,39}$ has a potential Donagi--Morrison lift $M$ of type $\g{6}{49}$. However, $\Lambda^2_{56,39}\cong\Lambda^6_{56,49}$, and so L3 does not hold. In this case, because $\rho(56,2,39)=-1$ and $\rho(56,3,44)=-4$, we clearly have $\mathcal{M}^2_{56,39}\nsubseteq\mathcal{M}^3_{56,44}$. Hence the failure of L3 does not necessarily obstruct our program to prove that \Cref{conj DM} implies \Cref{conjecture}.
		
	The next genus where L3 fails is $g=89$, where the locus $\mathcal{M}^3_{89,69}$ could possibly be contained in $\mathcal{M}^4_{89,75}$ or $\mathcal{M}^5_{89,79}$. This is because line bundles of type $\g{3}{36}$ and $\g{4}{75}$ have a potential Donagi--Morrison lift $M$ of type $\g{10}{85}$, and the lattice $\langle H,M\rangle=\Lambda^{10}_{89,85}$ is isomorphic to $\Lambda^3_{89,69}$, so that L3 does not hold. In this example, $\mathcal{M}^3_{89,69}$ has codimension $3$ in $\mathcal{M}_{89}$, whereas $\mathcal{M}^4_{89,75}$ and $\mathcal{M}^5_{89,79}$ both have codimension $1$, hence the codimensions of the loci do not rule out the possibility that $\mathcal{M}^3_{89,69}$ is not maximal. Thus in genus $89$, \Cref{conj DM} together with L2 is not sufficient to imply \Cref{conjecture} without additional techniques.
	
	We note that below genus $200$, except for genus $56,89,91,92,145,153$, and $190$, L2 holds, and thus \Cref{conj DM} implies \Cref{conjecture}.
\end{remark}

\begin{prop}\label{prop philosophy}
	Let $\mathcal{M}^r_{g,d}$ and  $\mathcal{M}^{r^\prime}_{g,d^\prime}$ be two expected maximal Brill--Noether loci. Suppose $(S,H)$ is  a polarized K3 surface with $\Pic(S)=\Lambda^r_{g,d}$, and $L\otimes\mathcal{O}_C$ is a $\g{r}{d}$. If the Donagi--Morrison conjecture (\Cref{conj DM}) holds for $\g{r^\prime}{d^\prime}$ on $C$ and L3 holds for the pair $(\Lambda^r_{g,d},\Lambda^{r^\prime}_{g,d^\prime})$, then $\mathcal{M}^r_{g,d}\nsubseteq \mathcal{M}^{r^\prime}_{g,d^\prime}$. In particular, if \Cref{conj DM} and L2 hold for all expected maximal $\g{r}{d}$ in genus $g$, then \Cref{conjecture} holds in genus $g$.
\end{prop}
\begin{proof}
	The condition L3 implies that $\Pic(S)$ cannot admit any potential Donagi--Morrison lift of the $\g{r^\prime}{d^\prime}$. Hence the existence of a $\g{r^\prime}{d^\prime}$ on $C$ contradicts the Donagi--Morrison conjecture. Therefore $C$ has no $\g{r^\prime}{d^\prime}$, as was to be shown.
\end{proof}

To state a related question, we need a simple definition.
\begin{defn}\label{defn special Cliff Index}
	We define the \emph{special Clifford index} of $C$ as $$\widetilde{\gamma}(C)\colonequals \min\{\gamma(A) ~\vert~ \rho(A)<0 \text{, } h^0(C,A) \geq 2\text{, and }h^1(C,A) \geq 2 \}.$$ We say a Brill--Noether special line bundle $A$ on $C$ \emph{computes} the special Clifford index if $\gamma(A)=\widetilde{\gamma}(C)$.
\end{defn}
Lelli-Chiesa's lifting result \cite[Theorem 4.2]{Lelli_Chiesa_2015} provides a lift of Brill--Noether special line bundles computing the Clifford index. A similar result for line bundles computing the special Clifford index of the curve together with L1 would imply that $\mathcal{M}^r_{g,d}\nsubseteq \mathcal{M}^{r^\prime}_{g,d^\prime}$ for $\gamma(r,d)\ge\gamma(r^\prime,d^\prime)$. We are left with three questions to which positive answers would imply parts of \Cref{conjecture}.

\begin{questionintro}
	When L1 or L2 fail, can the Brill--Noether loci be distinguished in another way?
\end{questionintro}

\begin{questionintro}
\label{question2}
	Under what conditions does a line bundle computing the special Clifford index of a curve $C$ lift to a line bundle on $S$?
\end{questionintro}

\begin{questionintro}
	Does the Donagi--Morrison conjecture hold for expected maximal $\g{r}{d}$s?
\end{questionintro}

 We note that the work on Brill--Noether theory for fixed gonality, if it were extended to higher rank, could provide another approach to distinguishing Brill--Noether loci that is complementary to the Donagi--Morrison lifting approach.

\section{Stability of Sheaves on K3 Surfaces}\label{Stability of Sheaves on K3 Surfaces}
We recall the notions of stability and slope stability of torsion free coherent sheaves on a polarized K3 surface $(S,H)$ and Harder--Narasimhan (HN) and Jordan--H\"{o}lder (JH) filtrations. Let $E$ be a torsion free coherent sheaf on $(S,H)$ The \emph{slope} of $E$ is $\mu_H(E)\colonequals \frac{c_1(E).H}{\operatorname{rk}(E)}$. A torsion free coherent sheaf is called slope stable or $\mu$-stable ($\mu$-semistable) if $\mu_H(F)<\mu_H(E)$ (respectively, $\mu_H(F)\le\mu_H(E)$) for all coherent sheaves $F\subseteq E$ with $0<\operatorname{rk}(F)<\operatorname{rk}(E)$. We define the \emph{normalized Hilbert polynomial} of $E$ to be \[p(E,n)\colonequals \frac{\chi(E\otimes H^{n})}{\operatorname{rk}(E)}=\frac{H^2}{2!}n^2+\mu(E)n+\frac{\chi(E)}{\operatorname{rk}(E)}\] where the second equality follows from Riemann--Roch. We say $E$ is (Gieseker) stable (semistable) if $p(F,n)<p(E,n)$ (respectively, $p(F,n)\le p(E,n)$) for all proper subsheaves $F\subsetneq E$, where for two polynomials $f(n), g(n)$ we say $f(n)<g(n)$ ($f(n)\le g(n)$) if this is true for $n\gg0$.

We have the following implications for a torsion free coherent sheaf $E$ \[\text{$\mu$-stable} \implies \text{ stable } \implies \text{ semistable } \implies \text{$\mu$-semistable}.\]

Every torsion free coherent sheaf $E$ has a unique Harder--Narasimhan filtration, which is an increasing filtration \[0=HN_0(E) \subset HN_1(E) \subset \cdots \subset HN_{\ell}(E)=E,\] such that the factors $gr_i^{HN}(E)=HN_i(E)/HN_{i-1}(E)$ for $i=1,\dots,\ell$ are torsion free semistable sheaves with normalized Hilbert polynomials $p_i=p(gr_i^{HN}(E),n)$ satisfying \[p_{\text{max}}=p_1>\cdots>p_{\ell}=p_{\text{min}}.\] In particular, we see that $\mu(gr_1^{HN}(E))>\mu(gr_2^{HN}(E))>\cdots > \mu(gr_{\ell}^{HN}(E))$. If $E$ is a vector bundle, the sheaves $HN_i(E)$ are locally free. We also have $\mu(HN_1(E))>\mu(HN_2(E))>\cdots > \mu(E)$.

Likewise, every $(\mu)$-semistable sheaf $E$ has a Jordan--H\"{o}lder filtration, which is an increasing filtration \[0=JH_0(E)\subset JH_1(E) \subset\cdots\subset JH_{\ell}(E)=E,\] such that the factors $gr_i^{JH}(E)=JH_i(E)/JH_{i-1}(E)$ for $i=1,\dots,\ell$ are torsion free stable sheaves with normalized Hilbert polynomial $p(E,n)$. In particular, $\mu(E)=\mu(gr_i^{JH}(E))$ for all $i$. The JH filtration is not uniquely determined, however the associated graded object $gr^{JH}(E)=\bigoplus\limits_i gr_i^{JH}(E)$ is uniquely determined by $E$.

We also briefly recall some facts about the moduli space of stable and semistable sheaves on K3 surfaces from \cite{huybrechts_2016}. For a sheaf $E$ on $(S,H)$, the \emph{Mukai vector} is given by \[v(E)\colonequals ch(E)\sqrt{td(S)}= \left(\operatorname{rk}(E),~ c_1(E),~ ch_2(E)+\operatorname{rk}(E)\right)=\left(\operatorname{rk}(E),~ c_1(E),~\chi(E)-\operatorname{rk}(E)\right),\] 
considered as an element in $H^\ast (S,\mathbb{Z})$. For a fixed Mukai vector $v$, the moduli space of semistable sheaves with Mukai vector $v$ is denoted $M(v)$, and the open (possibly empty) subscheme of stable sheaves is denoted $M(v)^s\subset M(v)$. The Mukai pairing is given by \[\langle v(E),v(F) \rangle \colonequals-\chi(E,F)=-\sum_i(-1)^i \operatorname{Ext}^i(E,F)=-\int_S v(E)^\ast \wedge v(F),\] where for $v(E)=v^0+v^2+v^4\in H^i(S,\mathbb{Z})$ with $v^i\in H^i(S,\mathbb{Z})$, $v(E)^\ast\colonequals v^0-v^2+v^4$. We recall that the space of stable sheaves with Mukai vector $v$, $M(v)^s$ is either empty or a smooth quasi-projective variety of dimension $2+\langle v,v \rangle$.

\section{Lazarsfeld--Mukai Bundles and Lifting}\label{Section Lazardsfeld--Mukai Bundles and Lifting and Generalized LM Bundles}
We briefly recall some facts about Lazarsfeld--Mukai bundles (LM bundles) and state a few useful facts that motivate our proof. Let $\iota\colon C \hookrightarrow S$ be a smooth irreducible curve of genus $g$ in $S$ and $A$ a basepoint free line bundle on $C$ of type $\g{r}{d}$. We define a bundle $F_{C,A}$ on $S$ via the short exact sequence \[\xymatrix{0 \ar[r] & F_{C,A} \ar[r] & H^0(C,A) \otimes \mathcal{O}_S \ar[r]^-{ev} & \iota_\ast(A) \ar[r] & 0.
} \] Dualizing gives $E_{C,A}\colonequals F_{C,A}^\vee$ (the LM bundle associated to $A$ on $C$) sitting in the short exact sequence \[\xymatrix{0 \ar[r] & H^0(C,A)^\vee \otimes\mathcal{O}_S \ar[r] & E_{C,A} \ar[r] & \iota_\ast(\omega_C\otimes A^\vee) \ar[r] & 0;}\] whereby the following facts about the LM bundle $E_{C,A}$ are readily proved.
\begin{prop}\label{LM bundle props}
	Let $E_{C,A}$ be a LM bundle associated to a basepoint free line bundle $A$ of type $\g{r}{d}$ on $C\subset S$, then:
	\begin{itemize}
		\item $\det E_{C,A}=c_1(E_{C,A})=[C]$ and $c_2(E_{C,A})=\operatorname{deg}(A)$;
		\item $\operatorname{rk}(E_{C,A})=r+1$ and $E_{C,A}$ is globally generated off the base locus of $\iota_\ast(\omega_C\otimes A^\vee)$;
		\item $h^0(S,E_{C,A})=h^0(C,A)+h^0(C,\omega_C\otimes A^\vee)=2r+1+g-d=g-(d-2r)+1$;
		\item $h^1(S,E_{C,A})=h^2(S,E_{C,A})=0$, $h^0(S,E^\vee_{C,A})=h^1(S,E^\vee_{C,A})=0$;
		\item $\chi(F_{C,A}\otimes E_{C,A})=2(1-\rho(g,r,d))$.
	\end{itemize}
\end{prop}
A vector bundle $E$ is called \emph{simple} if $\operatorname{End}(E)$ is a division algebra. Over an algebraically closed field, this is equivalent to $h^0(E^\vee \otimes E)=1$. Thus we see that $E_{C,A}$ is non-simple if $\rho(g,r,d)<0$.

In \cite{Lelli_Chiesa_2015}, generalized LM bundles are defined and prove useful in lifting special line bundles on a curve $C\in|H|$ to a line bundle on the polarized K3 surface $(S,H)$.
\begin{defn}
	Let $C$ be a curve and $A\in\Pic(C)$. The linear system $|A|$ is called \emph{primitive} if both $A$ and $\omega_C\otimes A^\vee$ are basepoint free.
\end{defn}
\begin{defn}[\cite{Lelli_Chiesa_2015} Definition 1]\label{gLM bundle defn}
	A torsion free coherent sheaf $E$ on $S$ with $h^2(S,E)=0$ is called a \emph{generalized Lazarsfeld--Mukai bundle} (gLM bundle) of type (I) or (II), respectively, if
	\begin{enumerate}[label=(\Roman*)]
		\item $E$ is locally free and generated by global sections off a finite set;\\
		or
		\item $E$ is globally generated.
	\end{enumerate}
\end{defn}
\begin{remark}[\cite{Lelli_Chiesa_2015} Remark 1]\label{remark glm is lm bundle}
	If conditions (I) and (II) of \Cref{gLM bundle defn} are both satisfied, then $E$ is the LM bundle associated with a smooth irreducible curve $C\subset S$ and a primitive linear series $(A,V)$ on $C$, i.e. $E=E_{C,(A,V)}$, where $E_{C,(A,V)}$ is the dual of the kernel of the evaluation map $V\otimes \mathcal{O}_S\to A$. Furthermore, $V=H^0(C,A)$ if and only if $h^1(S,E)=0$, in which case $E$ is just the LM bundle associated to $A$.
\end{remark}

\begin{defn}
	Let $E$ be a gLM bundle. The \emph{Clifford index of $E$} is:
	\[\gamma(E)\colonequals c_2(E) - 2(\operatorname{rk}(E)-1).\]
\end{defn}
\begin{remark}
	For the LM bundle $E_{C,A}$ for a smooth curve $C\subset S$ and $A$ a $\g{r}{d}$ on $C$, one has $\gamma(E_{C,A})=\gamma(A)$ by \Cref{LM bundle props}.
\end{remark}
\begin{lemma}[\cite{Lelli_Chiesa_2015} Corollary 2.5]\label{LC Cor 2.5}
	Let $E$ be a gLM bundle of rank $r$ and $c_1(E)^2>0$. Then, $\gamma(E)\ge 0$. Furthermore, $\gamma(E)=0$ only in the following cases:
	\begin{enumerate}[label={\normalfont(\alph*)}]
		\item $r=1$ and $E$ is a globally generated line bundle;
		\item $E=E_{C,\omega_C}$ for some smooth irreducible curve $C\subset S$ of genus $g=r\ge 2$;
		\item $r>1$ and $E=E_{C,(r-1)\g{1}{2}}$ for some smooth hyperelliptic curve $C\subset S$ of genus $g>r$.
	\end{enumerate}
\end{lemma}
\begin{lemma}\label{coker is gLM}
	Let $N\in \Pic(S)$ be nontrivial and globally generated with $h^0(S,N)\ne0$. Let $E=E_{C,A}$ and suppose we have a short exact sequence \[\xymatrix{0\ar[r] & N \ar[r] & E \ar[r] & E/N \ar[r] & 0}\] with $E/N$ torsion free. Then $E/N$ satisfies $h^1(S,E/N)=h^2(S,E/N)=0$. If $A$ is primitive, then $E/N$ is a gLM bundle of type (II). If we further assume that $E/N$ is locally free, then it is a LM bundle for a smooth irreducible curve $D\in|H-N|$. If $A$ is not primitive and $E/N$ is assumed locally free, then $E/N$ is a gLM bundle of type (I). In any of the above cases, we have
	\begin{itemize}
		\item $c_1(E/N)=H-N$;
		\item $c_2(E/N)=d+N^2-H. N$;
		\item $\gamma(E/N)=\gamma(E_{C,A})+N^2 - H. N+2$.
	\end{itemize}
\end{lemma}
\begin{proof}
	If $A$ is primitive, we see that $E/N$ is globally generated as $E$ is globally generated. From the long exact sequence in cohomology, and noting that $h^2(S,N)=h^1(S,E)=h^2(S,E)=0$, we see that $h^1(S,E/N)=h^2(S,E/N)=0$. Thus $E/N$ is a gLM bundle of type (II). If $E/N$ is assumed to be locally free, then as in \Cref{remark glm is lm bundle}, $E/N=E_{D,B}$ is the LM bundle associated to a smooth irreducible curve $D\subset S$ and a line bundle $B$ on $D$. Finally, if $A$ is not primitive, then $E/N$ is globally generated off a finite set as it is the quotient of $E$, which is also globally generated off a finite subset. Thus $E/N$ is a gLM of type (I).
	
	Applying Whitney's formula to the exact sequence, we see that \[1+c_1(E)+c_2(E)=\left(1+c_1(E/N)+c_2(E/N)\right)\left(1+N\right),\] hence $c_1(E/N)=H-N$ and $c_2(E/N)=d+N^2-H. N$. Finally, as $\gamma(E/N)=c_2(E/N)-2(\operatorname{rk}(E/N)-1)$ and  $\operatorname{rk}(E/N)=\operatorname{rk}(E)-1=(r+1)-1=r$, it follows that $$\gamma(E/N)=d+N^2-H. N - 2(r-1) = d-2r +N^2-H. N +2 = \gamma(E)+N^2-H. N +2.$$
\end{proof}
\begin{remark}
	If $A$ is of type $\g{r}{d}$ and $L=H-N$ is a lift of $A$ with $L^2=2r-2$, then the last equality gives $\gamma(E/N)=\gamma(A)+(2r-2)-d+2=0$.
\end{remark}
\begin{remark}\label{bounded c_2}
	The same argument as above shows that if $A$ is primitive and $M\subset E=E_{C,A}$ is a subsheaf such that $E/M$ is torsion free (e.g. obtained through a Harder--Narasimhan filtration), then $E/M$ is a gLM bundle of type (II). Moreover, by \cite[Proposition 2.7]{Lelli_Chiesa_2015}, if $c_1(E/M)^2=0$, then $c_2(E/M)=0$. In the following sections, we will use the contrapositive of this when $c_2(E/M)>0$. 
\end{remark}

We give a brief summary of gLM bundles of low Clifford index. Such a characterization can be useful in eliminating certain types of filtrations of Lazarsfeld--Mukai bundles, see Section~\ref{subsec:g18}.

\begin{prop}\label{Prop Clifford Index 1 and 2 gLM}
Let $E = E_{C,A}$ be a LM bundle associated to a primitive linear system $A$ on $C \subset S$. Suppose there is a globally generated saturated line bundle $N\subset E$ with $h^0(S,N)\ge 2$ and $\gamma(E/N)\le 2$. Then either $c_1(E/N)^2=0$ in which case $E/N=\mathcal{O}_S(\Sigma)^{\oplus r(A)}$ for an irreducible elliptic curve $\Sigma\subset S$, or $c_1(E/N)^2>0$ and one of the following holds:
	\begin{enumerate}[label=(\roman*)]
		\item $\gamma(E/N)=0$ (hence $E/N$ is a LM bundle);
		\item $(E/N)^{\vee\vee}$ is a LM bundle of Clifford index $0$;
		\item $E/N$ or $(E/N)^{\vee\vee}$ is a LM bundle of Clifford index $1$;
		\item $E/N$ is a LM bundle of Clifford index $2$.
	\end{enumerate}   
\end{prop}
\begin{proof}
	By \Cref{coker is gLM}, we see that $E/N$ is a gLM bundle of type (II). If $c_1(E/N)^2=0$, \cite[Proposition 2.7]{Lelli_Chiesa_2015} gives $E/N=\mathcal{O}_C(\Sigma)^{\oplus r(A)}$, as stated. 

We now assume $c_1(E/N)^2>0$.
If $\gamma(E/N)=0$, we are in case $(i)$. 

If $\gamma(E/N)=1$ and $E/N$ is locally free, we are in case $(iii)$. If $E/N$ is not locally free, then \cite[Proposition 2.4]{Lelli_Chiesa_2015} shows that $(E/N)^{\vee\vee}$ has Clifford index $0$, and we are in case $(ii)$.
	
	If $\gamma(E/N)=2$ and $E/N$ is locally free, we are in case $(iv)$. If $(E/N)$ is not locally free, then \cite[Proposition 2.4]{Lelli_Chiesa_2015} again shows that $(E/N)^{\vee\vee}$ has Clifford index $0$ or $1$, and we are in case $(ii)$ or $(iii)$, respectively.
\end{proof}

\begin{remark}\label{rem:bounding_invariants}
	Furthermore, in case $(iv)$ above, fixing the rank of $A$ narrows the possibilities for the classification of $E/N$. For example, when $A$ has rank $3$ and $E/N$ has Clifford index $2$, then $E/N=E_{D,\g{2}{6}}$ for a $\g{2}{6}$ on a smooth curve $D$ in the linear system of $\det(E/N)$.

	Likewise, restricting the Clifford index of a LM bundle $E$
        similarly restricts to which linear system $E$
        corresponds. For example, if $E$ is a LM bundle and
        $\gamma(E)=1$ or $\gamma(E)=2$, then a smooth irreducible
        curve $D\in|\det(E)|$ has $\gamma(D)\le2$ and is thus either
        hyperelliptic (when $\gamma(D)=0$), trigonal or a plane
        quintic (when $\gamma(D)=1$), or a plane sextic (when
        $\gamma(D)=2$ and $\text{rk}(E)=3$). 
	
	One could similarly characterize gLM bundles of type (II) of
        higher Clifford index, using \cite[Proposition
        2.4]{Lelli_Chiesa_2015} repeatedly as in \Cref{Prop Clifford
        Index 1 and 2 gLM}, and then fixing the rank as above. 
\end{remark}

We recall a few lemmas which show when a linear series on a curve $C\in |H|$ is the restriction of a line bundle $L$ on $S$.

\begin{lemma}\label{General 2r-2}
	Let $(S,H)$ be a polarized K3 surface of genus $g\ge2$, $C\in|H|$ be a smooth irreducible curve, and $L$ a globally generated line bundle on $S$ such that $L\vert_C$ is a $\g{r}{d}$ with $c_1(L).C=d<2g-2$. Then if $h^1(S,L)=0$, we have $L^2=2r-2-2h^1(S,L(-C))$.
\end{lemma}
\begin{proof}
	Since $H$ is basepoint free and $c_1(L(-C)).C=d-(2g-2)<0$, we have $h^0(S,L(-C))=0$, as in the proof of \cite[Proposition~2.1]{Knutsen2003}. We now consider the short exact sequence for a divisor $C\subset S$ tensored with $L$, \[\xymatrix{0  \ar[r] &L(-C) \ar[r]& L \ar[r] & L\vert_C \ar[r] & 0.}\] By Riemann-Roch on $C$ we have $h^1(S,L\vert_C)=h^1(C,L\vert_C)=r-d+g$, and as $h^1(S,L)=h^2(S,L)=0$, the long exact sequence in cohomology and Serre duality give $h^2(S,L(-C))=h^0(S,L(-C)^\vee)= r-d+g$. By Riemann-Roch on $S$, we have  \[ h^0(S,L(-C)^\vee)-h^1(S,L(-C))=2+\frac{c_1(L(-C))^2}{2}=2+\frac{c_1(L)^2-2d+2g-2}{2}= 1-d+g+\frac{c_1(L)^2}{2}\] thus $c_1(L)^2= 2r-2 -2h^1(S,L(-C))$.
\end{proof}
\begin{cor}\label{LC implies 2r-2}
	Let $(S,H)$ be a polarized K3 surface of genus $g\ge 2$, $A$ a complete $\g{r}{d}$ on a smooth $C\in |H|$. Let $N\in\Pic(S)$ be a line bundle with $h^0(S,N)\ge 2$ and $h^1(S,N)=0$. Assume $H\otimes N^\vee$ is globally generated, satisfies $h^1(S,H\otimes N^\vee)=0$, and is a lift of $A$. Then $c_1(H\otimes N^\vee)^2=2r-2$.
\end{cor}
\begin{proof}
	We have $h^1(S,N)=0$. Hence as $N^\vee=H\otimes N^\vee\otimes H^\vee$, Serre duality gives $0=h^1(S,N^\vee)=h^1(S,H\otimes N^\vee(-C))$. Thus \Cref{General 2r-2} shows that $(H-N)^2=2r-2$.
\end{proof}

\begin{remark}[\cite{Lelli_Chiesa_2015} Remark 6]\label{LC Remark 6}
	The proof \cite[Lemma 4.1]{Lelli_Chiesa_2015} shows that as soon as we have a nontrivial $N\in \Pic(S)$ with $h^0(S,N)\ne 0$ and an injection $N\hookrightarrow E_{C,A}$, we have $h^0(S,\iota_\ast(A)\otimes (H\otimes N^\vee)\otimes \mathcal{O}_C)=h^0(C,A^\vee \otimes (H\otimes N^\vee)\vert_C)\ne 0$, i.e., the linear series $|A|$ is contained in $|(H\otimes N^\vee)\vert_C|$. We also note that if $h^1(S,N)=0$, then \[H^0(C,(H\otimes N^\vee)\otimes \mathcal{O}_C)=H^0(S,H\otimes N^\vee)\vert_C.\]
\end{remark}

\begin{lemma}
	Let $N$ be a line bundle and $0\to N\to E\to E/N\to 0$ be a short exact sequence of coherent sheaves on a polarized K3 surface $(S,H)$, where $E/N$ is stable, $\operatorname{rk}(E)=r+1$, $c_1(E)=H$, $c_1(E)^2=2g-2\ge0$. If $h^0(S,N)<2$, then $c_2(E)\ge\frac{g(r-1)}{r}+\frac{2g-2}{r(r+1)}+r-\frac{1}{r}$.
\end{lemma}
\begin{proof}
	Since $\mu(N)\ge\mu(E)\ge 0$, we have $h^2(S,N)=0$. Therefore if $h^0(S,N)<2$ we have $c_1(N)^2\le -2$. Hence $$c_1(E/N)^2+2c_1(N).c_1(E/N)=c_1(E)^2-c_1(N)^2\ge 2g-2+2=2g$$ and $$c_1(E/N).c_1(N)=c_1(N).(c_1(E)-c_1(N))\ge \frac{2g-2}{r+1}+2,$$ where the last inequality comes from the fact that $\mu(N)\ge\mu(E)$. Thus $\frac{c_1(E/N)^2}{2}\ge g-c_1(N).c_1(E/N)$.
	
	Furthermore, since $E/N$ is stable of rank $r$, the dimension of the moduli space of stable sheaves with Mukai vector $\nu(E/N)$, $M^{s}_{\nu(E/N)}$, has non-negative dimension. Thus $2rc_2(E/N)-(r-1)c_1(E/N)^2-2(r^2-1)\ge 0$, and we have $c_2(E/N)\ge r-\frac{1}{r}+\left( \frac{r-1}{2r}\right)c_1(E/N)^2$. 
	
	We now calculate $c_2(E)=c_1(E/N).c_1(N)+c_2(E/N)\ge\frac{g(r-1)}{r}+\frac{2g-2}{r(r+1)}+r-\frac{1}{r}$, as desired.
\end{proof}

We present a version of \cite[Proposition 7.4]{Lelli_Chiesa_2013} which motivates our proof strategy below.
\begin{prop}\label{Prop Proof Strategy}
	Let $(S,H)$ be a polarized K3 surface and $A$ be a complete basepoint free $g^r_d$ on a smooth irreducible curve $C\in|H|$ with $r\ge 2$ and let $E=E_{C,A}$. Suppose that $E$ sits in a short exact sequence \[\xymatrix{0\ar[r] & N\ar [r]& E \ar[r] & E/N\ar[r] & 0}\] for some line bundle $N$ and $c_2(E)=d<\frac{g(r-1)}{r}+\frac{2g-2}{r(r+1)}+r-\frac{1}{r}$. If $E/N$ is stable, or $E/N$ is semistable and there are no elliptic curves on $S$, then $|A|$ is contained in the restriction to $C$ of the linear system $|H\otimes N^\vee|$ on $S$. Moreover, $H\otimes N^\vee$ is adapted to $|H|$ and $\gamma(H\otimes N^\vee \otimes\mathcal{O}_C)\le d-r-3 $.
\end{prop}
\begin{proof}
	By the previous lemma, $h^0(S,N)\ge 2$. We also have $h^0(S,\det E/N)\ge 2$ from \cite[Lemma 3.3]{Lelli_Chiesa_2013}. 
	We note that $(E/N)^{\vee\vee}$ is globally generated off a finite set and \[h^i(S,(E/N)^{\vee\vee})=h^i(S,E/N)=0 \text{ for } i=1,2.\] Since $\det E/N = \det(E/N)^{\vee\vee}$ is basepoint free and nontrivial, $\det E/N$ is nef, thus $c_1(E/N)^2\ge 0$. If $h^1(S,\det E/N)\ne 0$, then $c_1(E/N)^2=0$ by Saint-Donat. By \cite[Proposition 1.1]{GreenLaz}, there is a smooth elliptic curve $\Sigma\subset S$ such that $(E/N)^{\vee\vee}=\mathcal{O}(\Sigma)^{\oplus 3}$. This contradicts the stability of $E/N$ (or the non-existence of elliptic curves on $S$), thus we must have $c_1(E/N)^2\ge 2$ (hence $c_2(E/N)\ge r+1-\frac{2}{r}$) and $h^1(S,\det E/N)=0$. This ensures that $h^0(C,\det E/N\otimes \mathcal{O}_C)=h^0(C,H\otimes N^\vee \otimes\mathcal{O}_C)$ does not depend on the curve $C\in |H|_s$. Hence $\det E/N = H\otimes N^\vee$ is adapted to $|H|$. We calculate 
	\begin{align*}
		\gamma(\det E/N\otimes\mathcal{O}_C) &= c_1(E/N).c_1(E) - 2h^0(C,\det E/N \otimes\mathcal{O}_C) +2\\
		&= c_1(E/N)^2 + c_1(N).c_1(E/N)- 2h^0(C,\det E/N \otimes\mathcal{O}_C) +2\\
		&\le c_1(E/N)^2 - 2h^0(S,\det E/N) + c_1(N).c_1(E/N) +2\\
		&= -2h^1(S,\det E/N) -4 + c_1(N).c_1(E/N) +2\\
		&= d-c_2(E/N) -2 \le d-r-3.
	\end{align*}
	The claim that $|A|$ is contained in $|H\otimes N^\vee \otimes\mathcal{O}_C|$ is proved in the same way as in \cite[Lemma 4.1]{Lelli_Chiesa_2015}.
\end{proof}
\begin{remark}
	In the above proposition, if $A$ is of type $\g{3}{d}$, then $\gamma(H\otimes N^\vee \otimes\mathcal{O}_C)\le d-r-3=\gamma(A)$. However, as soon as $r\ge 4$, then $\gamma(H\otimes N^\vee \otimes\mathcal{O}_C)$ may be bigger than $\gamma(A)$. However, Lelli-Chiesa proves in \cite[Propositioon 5.1]{Lelli_Chiesa_2015} that $\gamma(H\otimes N^\vee\otimes\mathcal{O}_C)\le \gamma(A)$ whenever $N\subset E$ is a saturated subsheaf and $h^1(S,N)=0$.
\end{remark}

\section{Filtrations of Lazarsfeld--Mukai Bundles of Rank 4}\label{Section Filtrations of Lazarsfeld--Mukai Bundles of Rank 4}

Throughout this section, $(S,H)$ is a polarized K3 surface of genus $g$, $C \in |H|$ is a smooth irreducible curve, $A$ is a line bundle of type $\g{3}{d}$ on $C$, and $E=E_{C,A}$ is the LM bundle corresponding to $A$. Given $E$, we can take its JH filtration or take its HN filtration, further take JH filtrations of the properly semistable factors, lift the JH factors and expand the HN filtration of $E$ to arrive at a \emph{terminal filtration} such that all quotients are stable sheaves. We enumerate all the possibilities listing a filtration by the ranks of the terms, i.e., a filtration of type $1\subset 4$ is a filtration $0\subset N \subset E$ where $\operatorname{rk}(N)=1$.

The terminal filtrations correspond to flags of $E$ where each quotient is stable, hence the terminal filtrations are
\begin{align*}
& 1\subset4, \quad
2\subset4, \quad
3\subset4,\\
& 1\subset2\subset4, \quad
1\subset3\subset4, \quad
2\subset3\subset4,\\
& 1\subset2\subset3\subset4.
\end{align*}

In order to apply \Cref{Prop Proof Strategy}, we want to show that given the $\g{3}{d}$, $E$ must have a terminal filtration of type $1\subset4$. In all other cases, we want to find a lower bound on $d=c_2(E)$. To this end, we find a bound for $c_2(E)$ in terms of the intersections of the Chern roots of the LM bundle $E$. We begin by noting a few general bounds, and then deal with each filtration.

We slightly generalize the proof of \cite[Lemma 4.1]{Lelli_Chiesa_2013}.
\begin{prop}\label{LC rank 3 lemma 4.1}
	Let $E$ a LM bundle with $c_1(E)=H$ and $\mu(E)=\frac{g-1}{2}>0$ sitting in an exact sequence \[\xymatrix{0 \ar[r] & M \ar[r] & E \ar[r] & M_1\ar[r] & 0}\] where $M$ and $M_1$ are coherent sheaves. Suppose that the general smooth curve $C\in|H|$ has (constant) Clifford index $\gamma=\gamma(C)$. Then one has $c_1(M).c_1(M_1)\ge \gamma+2$.
\end{prop}

\begin{proof}
	We write $\mu(F)=\mu_H(F)$. Since $M_1$ is a quotient of $E$, it is globally generated off a finite set of points. Moreover, we have $h^2(S,M_1)=0$, thus $h^0(S,\det M_1)\ge 2$ by \cite[Lemma 3.3]{Lelli_Chiesa_2013} as the vector bundle $M_1^{\vee\vee}$ is globally generated off a finite number of points and $\det(M_1)\colonequals\det(M_1^{\vee\vee})$. As in \cite[Lemma 3.2]{Lelli_Chiesa_2013}, we see that $\det M_1$ is basepoint free and nontrivial, thus $\mu(\det M_1)>0$, $\mu(M)>0$. Hence as $\mu(\det M)\ge\mu(M)>0$, \cite[Proposition 3.1]{Lelli_Chiesa_2013} shows that $h^2(S,\det M_1)=0$, $h^2(S,\det M)=0$, and that $\det M_1$ is nef whereby $c_1(M_1)^2\ge 0$. 
	
	
	Furthermore, as \[\mu(M)=\frac{c_1(M).c_1(E)}{\operatorname{rk}(M)}=\frac{c_1(M).(c_1(M)+c_1(M_1))}{\operatorname{rk}(M)}\ge \frac{g-1}{2},\] we have $c_1(M).c_1(M_1)\ge \operatorname{rk}(M)\frac{g-1}{2}-c_1(M)^2$. Since $h^2(S,\det M)=0$, we note that \[h^0(S,\det M)\ge h^0(S,\det M)-h^1(S,\det M) = \chi(\det M)=2+\frac{c_1(M)^2}{2}.\] Therefore, if $2>h^0(S,\det M)$, then $c_1(M)^2\le -2$, and thus $$c_1(M).c_1(M_1)\ge \operatorname{rk}(M)\frac{g-1}{2}+2 \ge \operatorname{rk}(M)\gamma+2\ge \gamma+2$$ as $\operatorname{rk}(M) \ge 1$.
	
	Hence from now on we assume that $h^0(S,\det M)\ge 2$. Since $\omega_C\otimes (\det M_1)^\vee\otimes\mathcal{O}_C=\det M \otimes\mathcal{O}_C$, the line bundle $\det M_1\otimes\mathcal{O}_C$ contributes to $\gamma(C)$. Tensoring the short exact sequence for $\mathcal{O}_C$ with $\det M_1$ gives \[\xymatrix{0\ar[r] & \det M^\vee \ar[r] & \det M_1 \ar[r] & \det M_1\otimes\mathcal{O}_C \ar[r] & 0},\] which gives $h^0(C,\det M_1 \otimes\mathcal{O}_C)\ge h^0(S,\det M_1)$. It follows that 
	\begin{align*}
		\gamma(\det M_1 \otimes\mathcal{O}_C) &= c_1(M_1).(c_1(M)+c_1(M_1)) -2h^0(C,\det M_1\otimes\mathcal{O}_C)+2\\
		&\le c_1(M_1)^2 + c_1(M).c_1(M_1) -2\chi(\det M_1)-2h^1(S,\det M_1)+2\\
		&= -2 + c_1(M).c_1(M_1)-2h^1(S,\det M_1).
	\end{align*}
	By assumption, we have $\gamma(\det M_1\otimes\mathcal{O}_C)\ge \gamma$, thus $c_1(M).c_1(M_1)\ge \gamma+2+2h^1(S,\det M_1)\ge \gamma+2$, as desired.
\end{proof}

\begin{remark}
	It follows from the second half of the proof that if $M$ and $M_1$ are coherent sheaves such that $c_1(M)+c_1(M_1)=c_1(E)$, $\det M_1\otimes\mathcal{O}_C$ (hence also $\det M\otimes\mathcal{O}_C$) contributes to $\gamma(C)$, and $h^2(S,\det M_1)=0$ (or $h^2(S,\det M)=0$), then $c_1(M).c_1(M_1)\ge \gamma(C)+2+2h^1(S,\det M_1)\ge \gamma(C)+2$ (or $c_1(M).c_1(M_1)\ge \gamma(C)+2+2h^1(S,\det M)\ge \gamma(C)+2$).
\end{remark}

\begin{prop}\label{Prop quotients contribute to Cliff(C)}
	Let $(S,H)$ be a polarized K3 surface, $C\in|H|$ a smooth irreducible curve, $A$ a basepoint free line bundle on $A$ of type $\g{3}{d}$, and $E=E_{C,A}$. Suppose $E$ sits in an exact sequence \[\xymatrix{0\ar[r] & M \ar[r] & E \ar[r] & E/M \ar[r] & 0},\] where $M$ and $E/M$ are coherent torsion free sheaves on $S$ and $\mu(M)\ge\mu(E)\ge\mu(E/M)$. If $\operatorname{rk} (M)\ge \operatorname{rk}(E/M)$, then $c_1(M)^2\ge c_1(E/M)^2$. And if $\operatorname{rk} (M)> \operatorname{rk}(E/M)$, then $c_1(M)^2> c_1(E/M)^2$. In particular,  $\det(E/M)\otimes\mathcal{O}_C$ contributes to $\gamma(C)$.
\end{prop}
\begin{proof}
	As in \Cref{LC rank 3 lemma 4.1}, we see $h^0(S,\det E/M)\ge 2$, $\mu(E/M)>0$, $\det(E/M)$ is nef, and $h^2(S,\det M)=0$. Since $h^0(S,\det E/M)\ge 2$, it remains to show that $h^0(S,\det M)\ge 2$. 
	
	We observe that \[c_1(M)^2+c_1(M).c_1(E/M)=\operatorname{rk}(M) \mu(M)\ge \operatorname{rk}(E/M)\mu(E/M)=c_1(E/M)^2+c_1(M).c_1(E.M)\] whence $c_1(M)^2\ge c_1(E/M)^2\ge 0$ as $\det(E/M)$ is nef. 
	
	Since $h^2(S,\det M)=0$, we have $h^0(S,\det M)\ge\chi(\det M)= 2+\frac{c_1(M)^2}{2}$. Thus as $c_1(M)^2\ge 0$, $\det(E/M)\otimes\mathcal{O}_C$ contributes to $\gamma(C)$.
\end{proof}

For each terminal filtration not of the form $0\subset 1\subset 4$, we find a lower bound for $d=c_2(E)$. That is whenever $E$ does not have a maximal destabilizing sub-line bundle, we find that $d$ must be large. In effect, $c_2(E)$ controls the complexity of its Harder--Narasimhan and Jordan--H\"{o}lder filtrations.

\subsection{Filtration \texorpdfstring{$2\subset 4$}{}}

We assume $E$ is unstable with terminal filtration $0\subset M \subset E$ with $M$ and $M_1=E/M$ stable rank $2$ torsion free sheaves. Thus $E$ sits in an exact sequence of the form \[\xymatrix{0\ar[r] & M \ar[r] & E \ar[r] & M_1\ar[r]& 0}.\] We have 
\begin{gather}
	\mu(M)\ge \mu(E)=\frac{g-1}{2}\ge\mu(M_1)\\
	d=c_2(E)=c_1(M).c_1(M_1)+c_2(M)+c_2(M_1)
\end{gather}

\begin{lemma}\label{lemma bound 2<4 filt}
	Suppose $C\in|H|_s$ has Clifford index $\gamma=\gamma(C)$. Then if $E$ is as above, we have $d\ge \frac{\gamma}{2}+4+\frac{g-1}{2} $.
\end{lemma}
\begin{proof}
	From \Cref{LC rank 3 lemma 4.1} and \Cref{Prop quotients contribute to Cliff(C)}, we see that $c_1(M).c_1(M_1)\ge \gamma+2$. Stability of $M$ and $M_1$ give $-2\le\langle \nu(M_{(1)}),\nu(M_{(1)})\rangle = 4c_2(M_{(1)})-c_1(M_{(1)})^2-8$, thus $c_2(M_{(1)})\ge\frac{3}{2}+\frac{c_1(M_{(1)})^2}{4}$. 
	
	We have \[\frac{c_1(M)^2+c_1(M_1)^2}{4}+ \frac{c_1(M).c_1(M_1)}{2} =\frac{\mu(M)+\mu(M_1)}{2}=\frac{(c_1(M)+c_1(M_1))^2}{4}=\mu(E)=\frac{g-1}{2}.  \]
	We now calculate
	\begin{align*}
		d&= c_1(M).c_1(M_1)+c_2(M)+c_2(M_1)\\
		&\ge c_1(M).c_1(M_1)+3+\frac{c_1(M)^2+c_1(M_1)^2}{4}\\
		&= c_1(M).c_1(M_1) + 3 + \frac{g-1}{2} -\frac{c_1(M).c_1(M_1)}{2}\\
		&\ge \frac{\gamma+2}{2}+3+\frac{g-1}{2}.
	\end{align*}
	as claimed.
\end{proof}

\subsection{Filtration \texorpdfstring{$3\subset 4$}{}}

We assume $E=E_{C,A}$ is unstable with terminal filtration $0\subset M\subset E$ with $M$ a stable rank $3$ torsion free sheaf. Thus $E$ sits in an extension \[\xymatrix{0\ar[r] & M \ar[r] & E\ar[r] & N\otimes I_{\xi} \ar[r] & 0}\] where $N$ is a line bundle and $I_\xi$ is the ideal sheaf of a $0$-dimensional subscheme $\xi\subset S$ of length $l(\xi)=d-c_1(M).c_1(N)$. We have 
\begin{gather}
	\mu(M)\ge \mu(E)=\frac{g-1}{2}\ge \mu(N)\\
	c_1(H)=c_1(E)=c_1(M)+c_1(N)\\
	d=c_2(E)=c_1(N).c_1(M)+c_2(M)+l(\xi)
\end{gather}
\begin{lemma}
	Suppose $C\in|H|_s$ has Clifford index $\gamma=\gamma(C)$. Then if $E$ is as above, we have $d\ge \frac{2}{3}(\gamma+2)+\frac{g}{2} + \frac{13}{6}$.
\end{lemma}
\begin{proof}
	From \Cref{LC rank 3 lemma 4.1} and \Cref{Prop quotients contribute to Cliff(C)}, we see that $c_1(N).c_1(M)\ge \gamma+2$.
	
	As $M$ is stable, we have $-2\le \langle \nu(M),\nu(M)\rangle = 6c_2(M)-2c_1(M)^2-18$, thus $c_2(M)\ge \frac{8+c_1(M)^2}{3}$. Thus \begin{align*}
		d&= c_1(N).c_1(M)+c_2(M)+l(\xi)\\
		&\ge c_1(N).c_1(M)+\frac{c_1(M)^2}{3} + \frac{8}{3}\\
		&\ge c_1(N).c_1(M) +\frac{g-1}{2}-\frac{c_1(N).c_1(M)}{3}+\frac{8}{3}\\
		&\ge \frac{2}{3}(\gamma+2)+\frac{g}{2}+\frac{13}{6},
	\end{align*} as desired.
\end{proof}

\subsection{Filtration \texorpdfstring{$1\subset2\subset 4$}{}}

We assume $E$ has a terminal filtration $0\subset N \subset M \subset E$ with $\operatorname{rk}(N)=1$, $\operatorname{rk}(M)=2$, and $E/M=M_1$ a stable torsion free sheaf. Furthermore, we have 
\begin{gather}
	\mu(N)\ge \mu(M) \ge \mu(E)=\frac{g-1}{2} \ge \mu(M_1)\\
	\mu(M)\ge \mu(M/N)\ge \mu(E/N)\ge \mu(M_1)\\
	d= c_2(E)=c_2(M)+c_2(M_1)+c_1(M).c_1(M_1)= c_1(N).c_1(M/N)+c_1(N).c_1(M_1)+c_1(M/N).c_1(M_1)+c_2(M_1)
\end{gather}
Moreover, as $M_1$ is stable, we have \[-2 \le \langle \nu(M_1),\nu(M_1)\rangle = c_1(M_1)^2 -4\chi(M_1)+8 = 4c_2(M_1)-c_1(M_1)^2-8\] thus $c_2(M_1)\ge \frac{3}{2}+ \frac{c_1(M_1)^2}{4}$. Therefore we have 
\begin{gather}
	d\ge \frac{3}{2}+\frac{c_1(M_1)^2}{4} +  c_1(N).c_1(M/N)+c_1(N).c_1(M_1)+c_1(M/N).c_1(M_1).
\end{gather}

\begin{lemma}
	Suppose $E$ is as above. Then $\det M_1\otimes\mathcal{O}_C$ contributes to $\gamma(C)$ and one of the following occurs:
	\begin{enumerate}[label={\normalfont(\alph*)}]
		\item $N\otimes\mathcal{O}_C$ and $(M/N)\otimes\mathcal{O}_C$ contribute to $\gamma(C)$;
		\item $c_1(N).(c_1(M_1)+c_1(M/N))\ge \frac{g-1}{2}+2$ and either $(M/N)\otimes\mathcal{O}_C$ contributes to $\gamma(C)$ or \newline$c_1(M/N).(c_1(N)+c_1(M_1))\ge g$ ;
		\item $N\otimes\mathcal{O}_C$ contributes to $\gamma(C)$ and $c_1(M/N).(c_1(N)+c_1(M_1))\ge 2+\frac{c_1(M).c_1(M_1)}{2}+\frac{c_1(M_1)^2}{2}$;
		\item $c_1(N).c_1(M/N)\ge \frac{g+3}{2}$. 
	\end{enumerate}
\end{lemma}
\begin{proof}
	From \Cref{LC rank 3 lemma 4.1} and \Cref{Prop quotients contribute to Cliff(C)}, we see that $\det M_1\otimes\mathcal{O}_C$ contributes to $\gamma(C)$.
	
	We have the following four cases:
	\begin{enumerate}[label={\normalfont(\roman*)}]
		\item $h^0(S,M/N), h^0(S,N)\ge 2$
		\item $h^0(S,M/N)\ge2$ and $h^0(S,N)< 2$
		\item $h^0(S,M/N)< 2$ and $h^0(S,N)\ge2$ 
		\item $h^0(S,M/N), h^0(S,N)< 2$
	\end{enumerate}
	
	In case (i), we have $h^0(S,H\otimes(M/N)^\vee)=h^0(S,\det M_1 \otimes N)\ge 2$ and $h^0(S,H\otimes N^\vee)=h^0(S,\det M_1 \otimes M/N)\ge 2$ as $\det M_1$ has global sections. Thus we are in case (a) of the lemma.
	
	In case (ii), we see that $\chi(N)<2$, hence $c_1(N)^2\le -2$, and we calculate 
	\begin{align*}
		c_1(N).(c_1(M_1).c_1(M/N))&=c_1(N).(c_1(E)-c_1(N))\\
		&= \mu(N) - c_1(N)^2 \ge \mu(E) +2 = \frac{g-1}{2}+2,
	\end{align*} thus the first statement of case (b) is proved. We now observe that $c_1(N\otimes \det M_1)^2> c_1(M/N)^2$ which follows from the computation $c_1(N\otimes \det M_1)^2- c_1(M/N)^2\ge 2\mu(M_1)>0$.

	If $c_1(N\otimes \det M_1)^2<0$, then also $c_1(M/N)^2<0$, and we calculate 
	\begin{align*}
		2g-2=c_1(E)^2&= (c_1(N)+c_1(M/N)+c_1(M_1))^2\\
		&= c_1(N\otimes \det M_1)^2+2c_1(N\otimes \det M_1).c_1(M/N)+c_1(M/N)^2\\
		&<2(c_1(N)+c_1(M_1)).c_1(M/N),
	\end{align*}
	thus $c_1(M/N).(c_1(N)+c_1(M_1))\ge g$. Else $c_1(N\otimes \det M_1)^2\ge 0$ and so $h^0(S,H\otimes (M/N)^\vee)=h^0(S,N\otimes \det M_1)\ge 2$ and so $M/N$ contributes to $\gamma(C)$. Thus we are in case (b).
	
	In case (iii), since $\det E/N \cong \det M_1\otimes M/N$, we have $h^0(S,\det M_1\otimes M/N)\ge 2$. Thus as $h^0(S,N)\ge 2$, we see that $N\otimes\mathcal{O}_C$ contributes to $\gamma(C)$. Therefore, as $h^0(S,M/N)<2$, we have $c_1(M/N)^2\le -2$. 	
	
	In cases (iii) and (iv), we have $c_1(M/N)^2\le -2$. We now calculate 
	\begin{align*}
		2g-2=c_1(E)^2&=c_1(M/N)^2+c_1(N)^2+c_1(M_1)^2+2c_1(M/N).c_1(N)+2c_1(M/N).c_1(M_1)+2c_1(N).c_1(M_1)\\
		&\le c_1(N)^2+c_1(M_1)^2+2c_1(M/N).c_1(N)+2c_1(M/N).c_1(M_1)+2c_1(N).c_1(M_1)-2\\
		&\le c_1(N)^2+g-3+2c_1(M/N).c_1(N),
	\end{align*}
	thus \begin{equation}\label{124 filtration cases iii iv calculation}
		c_1(N).c_1(M/N)\ge\frac{g+1}{2}-\frac{c_1(N)^2}{2}.
	\end{equation}
	
	In case (iii), we observe that since $$c_1(M/N).(c_1(N)+c_1(M_1))+c_1(M/N)^2=\mu(M/N)\ge\mu(E/N)=\frac{(c_1(E/N)).(c_1(E))}{3},$$
	we have \begin{align*}
		c_1(M/N).(c_1(N)+c_1(M_1))\ge -c_1(M/N)^2+\frac{c_1(M/N)^2}{3}&+\frac{c_1(M/N).(c_1(N)+c_1(M_1))}{3}\\&+\frac{c_1(M).c_1(M_1)}{3}+\frac{c_1(M_1)^2}{3}.
	\end{align*}
	And subtracting $c_1(M/N).(c_1(N)+c_1(M_1))/3$ from both sides and multiplying by $3/2$ yields 
	\begin{align*}
		c_1(M/N).(c_1(N)+c_1(M_1))\ge -c_1(M/N)^2+\frac{c_1(M).c_1(M_1)}{2}+\frac{c_1(M_1)^2}{2}.
	\end{align*} 
	Noting that $c_1(M/N)^2\le -2$ shows we are in case (c).

	In case (iv), as $h^0(S,N),h^0(S,M/N)<2$, we have $c_1(N)^2,c_1(M/N)^2\le -2$, thus \Cref{124 filtration cases iii iv calculation} gives $c_1(N).c_1(M/N)\ge\frac{g+1}{2}-\frac{c_1(N)^2}{2}\ge\frac{g+1}{2}+1= \frac{g+3}{2}$, and we are in case (d).
\end{proof}

\begin{lemma}\label{general proof of d ge BLA}
	With $E$ as above, if general curves in $|H|_s$ have Clifford index $\gamma=\gamma(C)$, and $m=D^2$ is the minimum self-intersection of an effective classes $D\in\Pic(S)$ (i.e. there are no curves of genus $g^\prime<\frac{m+2}{2}$ on $S$), then we have $d\ge \frac{5}{4}\gamma+\frac{m}{2}+5$ or $d\ge5+\frac{3}{2}\gamma$. Moreover, when $A$ is primitive, then we can assume $m\ge 2$.
\end{lemma}
\begin{proof}
	We write $2d\ge 3+\frac{c_1(M_1)^2}{2}+c_1(N).c_1(E/N)+c_1(M/N).(c_1(N)+c_1(M_1))+c_1(M).c_1(M_1)$, and apply bounds to each of the terms. From \Cref{LC rank 3 lemma 4.1}, we see that $c_1(N).c_1(E/N)\ge\gamma+2$, and $c_1(M).c_1(M_1)\ge\gamma+2$. In cases (a), (b), we have $c_1(M/N).(c_1(N)+c_1(M_1))\ge\gamma+2$. In case (c), we have $d\ge\frac{5}{4}\gamma+\frac{m}{2}+5$. Finally, in case (d), we have $d\ge 2+c_1(N).c_1(M/N)+c_1(M).c_1(M_1)\ge 2+\frac{g+13}{2}+\gamma+2$. And in any case, we have the desired inequality.
	
	When $A$ is primitive, $M_1$ is a gLM of type (II), and as $c_1(M_1)^2\ge 0$ we have $c_2(M_1)>0$, thus we cannot have $c_1(M_1)^2=0$ by \Cref{bounded c_2}. Therefore $m$ can be taken to be at least $2$. 
\end{proof}

\subsection{Filtration \texorpdfstring{$1\subset3\subset 4$}{}}

We assume $E$ has a terminal filtration $0\subset N \subset M \subset E$ with $\operatorname{rk}(N)=1$, $\operatorname{rk}(M)=3$, and $M/N$ a stable torsion free sheaf, and we call $E/M=N_1$. Furthermore, we have \begin{gather}
	\mu(N)\ge \mu(M)\ge\mu(E)\ge\mu(E/N)\ge\mu(N_1)\\
	\mu(M)\ge\mu(M/N)\ge\mu(E/N)\\
	d=c_2(E)=c_2(M/N)+c_1(M/N).c_1(N)+c_1(N).c_1(N_1)+c_1(N_1).c_1(M/N)
\end{gather}
Moreover, since $M/N$ is stable, we have \[-2\le\langle\nu(M/N),\nu(M/N)\rangle = c_1(M/N)^2-4\chi(M/N)+8=4c_2(M/N)-c_1(M/N)^2-8\] thus $c_2(M/N)\ge\frac{3}{2}+\frac{c_1(M/N)^2}{4}$.

\begin{lemma}
	Suppose $E$ is as above. Then $N_1\otimes\mathcal{O}_C$ contributes to $\gamma(C)$, and one of the following occurs:
	\begin{enumerate}[label={\normalfont(\alph*)}]
		\item $N\otimes\mathcal{O}_C$ and $\det (M/N)\otimes\mathcal{O}_C$ contribute to $\gamma(C)$;
		\item $c_1(N).(c_1(N_1)+c_1(M/N))\ge\frac{g+3}{2}\ge \gamma(C)+2$ and either $\det (M/N)\otimes\mathcal{O}_C$ contributes to $\gamma(C)$ or $\frac{c_1(M/N)^2}{2}+c_1(M/N).(c_1(N)+c_1(N_1))\ge g$;
		\item $N\otimes\mathcal{O}_C$ contributes to $\gamma(C)$ and $\frac{c_1(M/N)^2}{2} + c_1(M/N).c_1(N) \ge \frac{1}{2}c_1(N).(c_1(N_1)+c_1(M/N))$;
		\item $\frac{c_1(M/N)^2}{2}+c_1(M/N).c_1(N) \ge g+1$.
	\end{enumerate}
\end{lemma}
\begin{proof}
	From \Cref{LC rank 3 lemma 4.1} and \Cref{Prop quotients contribute to Cliff(C)}, we see that $N_1\otimes\mathcal{O}_C$ contributes to $\gamma(C)$ and $h^2(S,\det M/N)=h^2(S,M/N)=h^2(S,N)=0$. 
	
	We have the following four cases: 
	\begin{enumerate}[label={\normalfont(\roman*)}]
		\item $h^0(S,\det M/N), h^0(S,N)\ge 2$
		\item $h^0(S,\det M/N)\ge2$ and $h^0(S,N)< 2$
		\item $h^0(S,\det M/N)< 2$ and $h^0(S,N)\ge2$ 
		\item $h^0(S,\det M/N), h^0(S,N)< 2$
	\end{enumerate}
	In case (i), we have $h^0(S,H\otimes N^\vee)=h^0(S,\det M/N\otimes N_1)\ge2$, and $
	h^0(S,H\otimes\det M/N^\vee)=h^0(S,N\otimes N_1)\ge 2$ as $\det M/N$, $N$, and $N_1$ have global sections. Thus we are in case (a) of the lemma.
	
	In case (ii), we see that $\chi(N)<2$, thus $c_1(N)\le -2$, and we calculate 
	\begin{align*}
		c_1(N).(c_1(N_1)+c_1(M/N))=& c_1(N).(c_1(E)-c_1(N))\\
		&= \mu(N)-c_1(N)^2 \ge \mu(E)+2 =\frac{g+3}{2}.
	\end{align*}
	\begin{itemize}
		\item If $c_1(N\otimes N_1)^2,c_1(M/N)^2\ge 0$, then $\det M/N$ contributes to $\gamma(C)$ as $h^0(S,N\otimes N_1)=h^0(S,H-\det M/N)\ge2$.
		
		\item If $c_1(N\otimes N_1)^2\ge 2$ and $c_1(M/N)^2< 0$, then as above $\det M/N$ contributes to $\gamma(C)$.
		
		\item If $c_1(N\otimes N_1)^2<0$ and $c_1(M/N)^2\ge 0$ then we cannot say if $\det M/N$ contributes to $\gamma(C)$ as above. However, we calculate \begin{align*}
			2g-2=c_1(E)^2=&(c_1(M/N)+c_1(N\otimes N_1))^2\\
			&= c_1(M/N)^2+2c_1(M/N).(c_1(N)+c_1(N_1))+c_1(N\otimes N_1)^2 \\
			&< c_1(M/N)^2+2c_1(M/N).(c_1(N)+c_1(N_1)),
		\end{align*}
		thus $\frac{c_1(M/N)^2}{2}+c_1(M/N).(c_1(N)+c_1(N_1))\ge g$.
		
		\item If $c_1(N\otimes N_1)^2,c_1(M/N)^2< 0$, then the same calculation as above yields $\frac{c_1(M/N)^2}{2}+c_1(M/N).(c_1(N)+c_1(N_1))\ge g$.
	\end{itemize}
	Thus we are in case (b) of the lemma.
	
	In case (iii), since $\det E/N = N_1\otimes \det M/N$, \cite[Lemma 3.3]{Lelli_Chiesa_2013} implies that $h^0(S,N_1\otimes \det M/N)\ge 2$. Thus since $h^0(S,N)\ge 2$, we see that $N\otimes\mathcal{O}_C$ contributes to $\gamma(C)$. Furthermore, as $c_1(M/N)^2+c_1(M/N).c_1(N)\ge c_1(N_1)^2+c_1(N_1).c_1(N)$ and $c_1(N_1)^2\ge 0 > c_1(M/N)^2$, we have $c_1(M/N)^2+ c_1(M/N).c_1(N)\ge c_1(N_1).c_1(N)$. Thus 
	\begin{align*}
		c_1(M/N)^2+c_1(M/N).c_1(N)-\frac{1}{2}(c_1(N).(c_1(N_1)+c_1(M/N)))&\ge c_1(M/N)^2 +\frac{c_1(M/N).c_1(N)}{2} - \frac{c_1(N).c_1(N_1)}{2}\\
		&\ge \frac{c_1(M/N)^2}{2},		
	\end{align*}
	thus \[\frac{c_1(M/N)^2}{2} + c_1(M/N).c_1(N) \ge \frac{1}{2}c_1(N).(c_1(N_1)+c_1(M/N)),\] and we are in case (c).
	
	In case (iv), we see that $c_1(N)^2, c_1(M/N)^2\le -2$. We calculate \begin{align*}
		2g-2 =& c_1(E)^2 = (c_1(N)+c_1(N_1)+c_1(M/N))^2\\
		&\le c_1(N_1)^2 +c_1(M/N)^2 +2c_1(N).c_1(N_1) +2c_1(N).c_1(M/N) + 2c_1(N_1).c_1(M/N) -2\\
		&\le g-1 +2c_1(N).c_1(M/N) +c_1(M/N)^2 -2,
	\end{align*}
	thus $\frac{c_1(M/N)^2}{2}+c_1(N).c_1(M/N)\ge g+1$, and we are in case (d).
\end{proof}
\begin{remark}
	From the second half of the proof of \Cref{LC rank 3 lemma 4.1}, we see that in the situation above, if $C\in|H|_s$ has Clifford index $\gamma=\gamma(C)$, and if $\det M/N$ contributes to $\gamma(C)$, then we have $c_1(M/N).(c_1(N)+c_1(N_1))\ge \gamma+2+2h^1(S,\det M/N)$.
\end{remark}
\begin{lemma}
	With $E$ as above, if general curves in $|H|_s$ have Clifford index $\gamma=\gamma(C)$, we have $d\ge\frac{3}{2}\gamma+5$.
\end{lemma}
\begin{proof}
	We first see that if $c_1(M/N)^2\ge0$, then we are in cases (a) or (b) of the above lemma. Furthermore, we have $c_2(M/N)\ge2$. Thus in case (a), we have \begin{align*}
		2d&\ge 2(c_2(M/N)+c_1(M/N).c_1(N)+c_1(N).c_1(N_1)+c_1(N_1).c_1(M/N))\\
		&\ge 4+2c_1(M/N).c_1(N)+2c_1(N).c_1(N_1)+2c_1(N_1).c_1(M/N)\\
		&= 4+c_1(M/N).(c_1(N)+c_1(N_1))+c_1(N).(c_1(N_1)+c_1(M/N))+c_1(N_1).(c_1(M/N)+c_1(N))\\
		&\ge 4+ 3(\gamma+2),
	\end{align*} where the last inequality comes from \Cref{LC rank 3 lemma 4.1}. Thus $d\ge \frac{3}{2}\gamma+5$. In case (b), we calculate as in case (a) and get $d\ge \frac{3}{2}\gamma+5$ or 
	\begin{align*}
		2d &\ge 2\left(c_1(N).c_1(N_1)+c_1(N).c_1(M/N)+c_1(N_1).c_1(M/N)+\frac{c_1(M/N)^2}{4}\right)\\
		&\ge g + c_1(N).c_1(M/N) + 2c_1(N).c_1(N_1) +c_1(N_1).c_1(M/N)\\
		&\ge g+ 2(\gamma+2),
	\end{align*} hence $d\ge \gamma+2+\frac{g}{2}>\frac{3}{2}\gamma+5$. 
	
	If $c_1(M/N)^2<0$, in case (d), we have 
	\begin{align*}
		d&\ge \frac{3}{2}+\frac{g+1}{2}+\frac{c_1(N).c_1(M/N)}{2}+c_1(M/N).c_1(N_1)+c_1(N).c_1(N_1)\\
		&\ge \frac{g+4}{2}+k+\frac{g+1}{2}-\frac{c_1(M/N)^2}{4}\\
		&\ge \gamma+2+g+\frac{7}{2}.
	\end{align*}
	
	If $0>c_1(M/N)^2\ge -6$, then $c_2(M/N)\ge 0$, thus $\chi(\det M/N)\le 1$. Therefore $h^1(S, \det M/N)+1\ge h^0(S,\det M/N)$. Calculating as above, we see that
	\begin{itemize}
		\item in case (a), we have $d\ge \frac{3}{2}\gamma+5$;
		\item in case (b), we have $d\ge \frac{3}{2}\gamma+5$ or $d\ge \gamma+\frac{7}{2}+\frac{g+2}{2}$; and,
		\item in case (c), we have $d\ge \frac{3}{2}\gamma+5$.
	\end{itemize} 
	
	If $c_2(M/N)< 0$, then the stability of $M/N$ implies that $c_1(M/N)^2\le -8$ and \[-2 \le \langle \nu(M/N),\nu(M/N)\rangle = c_1(M/N)^2 +8 -4\chi(M/N)\le -4\chi(M/N),\] whereby $\chi(M/N)\le 0$. We now consider inequalities associated with various filtrations that lead to the terminal $1\subset 3\subset 4$ filtration of $E$. 
	
	If the JH filtration of $E$ is $1\subset 3\subset 4$, then we have $p(E)=p(M/N)$, which gives an equality of normalized Euler characteristics \[\frac{\chi(M/N)}{2}=\frac{\chi(E)}{4}=\frac{g-\gamma +1}{4}.\] Thus $0\ge2\chi(M/N)=g-d+7$, and hence $d\ge g+7$. 
	
	If the HN filtration of $E$ is $0\subset M\subset E$ with $\operatorname{rk}(M)=3$ and $M$ properly semistable, then the JH filtration of $M$ is $0\subset N \subset M$. Hence $\mu(M/N)=\mu(M)$ and $\mu(M)>\mu(E)$. Thus \[\frac{c_1(M/N)^2}{2}+\frac{c_1(M/N).c_1(N\otimes N_1)}{2}=\mu(M/N)>\mu(E)=\frac{g-1}{2},\] hence
	\begin{align*}
		d&\ge \frac{3}{2}+\frac{c_1(M/N)^2}{4}+c_1(M/N).c_1(N\otimes N_1)+c_1(N).c_1(N_1)\\
		&\ge \frac{3}{2} + \frac{g-1}{2}-\frac{c_1(M/N)^2}{4} + \frac{c_1(M/N).(c_1(N)+c_1(N_1))}{2}\\
		&\ge \frac{3}{2} + \frac{g-1}{2}+\frac{c_1(N).(c_1(N_1)+c_1(M/N))}{2}+\frac{c_1(N_1).(c_1(N)+c_1(M/N))}{2}\\
		&\ge \frac{3}{2}+\frac{g-1}{2}+\gamma+1
	\end{align*}
	where the last inequality comes from the fact that $N_1$ contributes to $\gamma(C)$, and that in cases (a),(b), and (c) $c_1(N).(c_1(N_1)+c_1(M/N))\ge \gamma+2$.
	
	If the HN filtration of $E$ is $0\subset N \subset E$ with $E/N$ properly semistable and the JH filtration of $E/N$ is $0\subset \overline{M}\subset E/N$ with $\operatorname{rk}(\overline{M})=2$, then we have an equality of normalized Euler characteristics \[\frac{\chi(E)-\chi(N)}{3}=\frac{\chi(E/N)}{3}=\frac{\chi(\overline{M})}{2}=\frac{\chi(M/N)}{2}.\] Thus $\chi(E)=g-\gamma+1=\frac{3\chi(M/N)}{2}+\chi(N)$, where $\gamma=d-6$ is the Clifford index of the $\g{3}{d}$ on $C$. From the short exact sequence \[\xymatrix{0\ar[r] & N \ar[r] & E \ar[r] & E/N \ar[r] & 0},\] we have $\chi(N)=h^0(S,E)-h^0(S,E/N)\le g-\gamma-1$ as $h^0(S,E/N)\ge2$. Therefore \[g-\gamma+1=\chi(E)\le \frac{3\chi(M/N)}{2}+g-\gamma-1,\] and thus $2\le\frac{3}{2}\chi(M/N)\le0$, which is a contradiction. Thus this does not occur, and in all cases we have at least $d\ge \frac{3}{2}\gamma+5$, as claimed. 
\end{proof}

\subsection{Filtration \texorpdfstring{$2\subset3\subset 4$}{}}

We assume $E$ has a terminal filtration $0\subset N \subset M \subset E$ with $N$ a stable torsion free sheaf of rank $\operatorname{rk}(N) =2$, $\operatorname{rk} (M) =3$, and $N_1=E/M$ a line bundle. Furthermore, we have 
\begin{gather}
	\mu(N)\ge \mu(M)\ge\mu(E)=\frac{g-1}{2} \ge \mu(N_1)\\
	\mu(M)\ge\mu(M/N)\ge \mu(E/N)\ge \mu(N_1)\\
	d= c_2(E) = c_2(N)+c_1(N).c_1(M/N)+c_1(N).c_1(N_1) +c_1(M/N).c_1(N_1)
\end{gather}
Moreover, as $N$ is stable, we have $c_2(N)\ge \frac{3}{2}+\frac{c_1(N)^2}{4}$.

\begin{lemma}
	Suppose $E$ is as above. Then $N_1\otimes\mathcal{O}_C$ contributes to $\gamma(C)$ and one of the following occurs:
	\begin{enumerate}[label={\normalfont(\alph*)}]
		\item $(\det N)\otimes\mathcal{O}_C$ and $(M/N)\otimes\mathcal{O}_C$ contribute to $\gamma(C)$;
		\item  $c_1(N).(c_1(N_1)+c_1(M/N))\ge g+1$and either $(M/N)\otimes\mathcal{O}_C$ contributes to $\gamma(C)$ or $c_1(M/N).(c_1(N)+c_1(N_1))\ge g$;
		\item $(\det N)\otimes\mathcal{O}_C$ contributes to $\gamma(C)$, we can assume $c_1(N)^2\ge 0$ and $c_1(M/N).c_1(N)\ge \frac{1}{2}c_1(N).(c_1(M/N)+c_1(N_1))$;
		\item $c_1(N)^2\le-2$ and $\frac{c_1(N)^2}{2}+c_1(M/N).c_1(N)\ge \frac{g+1}{2}$.
	\end{enumerate}
\end{lemma}
\begin{proof}
	From \Cref{LC rank 3 lemma 4.1} and \Cref{Prop quotients contribute to Cliff(C)}, we see that $N_1\otimes\mathcal{O}_C$ contributes to $\gamma(C)$ and $h^2(S,\det N)=h^2(S,\det M)=h^2(S,M/N)=h^2(S,\det E/N)=0$. 	
	
	We have the following four cases:
	\begin{enumerate}[label={\normalfont(\roman*)}]
		\item $h^0(S,M/N), h^0(S,\det N)\ge 2$
		\item $h^0(S,M/N)\ge2$ and $h^0(S,\det N)< 2$
		\item $h^0(S,M/N)< 2$ and $h^0(S,\det N)\ge2$ 
		\item $h^0(S,M/N), h^0(S,\det N)< 2$.
	\end{enumerate}
	
	In case (i), as $N_1$ has global sections, and $H- c_1(M/N)=c_1(N)+c_1(N_1)$ and $H-c_1(N)=c_1(N_1)+c_1(M/N)$, we see that both $(\det N)\otimes\mathcal{O}_C$ and $(M/N)\otimes\mathcal{O}_C$ contribute to $\gamma(C)$, and we are in case (a).
	
	In case (ii), we have $\chi(N)<2$, hence $c_1(N)^2\le -2$, and we calculate
	\begin{align*}
		c_1(N).(c_1(N_1)+c_1(M/N))&=c_1(N).(c_1(E)-c_1(N))\\
		&= 2\mu(N)-c_1(N)^2\ge g-1+2=g+1
	\end{align*}
	We now observe that $c_1(\det N\otimes N_1)^2\ge c_1(M/N)^2$ which follows from the following calculation
	\begin{align*}
		c_1(\det N\otimes N_1)^2 -c_1(M/N)^2&= c_1(N)^2+ 2c_1(N).c_1(N_1)+ c_1(N_1)^2-c_1(M/N)^2\\
		&= 2\mu(N)+\mu(N_1)-\mu(M/N)\ge \mu(N)+\mu(N_1)>0.
	\end{align*}
	If $c_1(\det N \otimes N_1)^2<0$, then also $c_1(M/N)<0$, and we calculate 
	\begin{align*}
		2g-2=c_1(E)^2&= (c_1(N)+c_1(M/N)+c_1(N_1))^2\\
		&= c_1(\det N \otimes N_1)^2 +2c_1(\det N \otimes N_1).c_1(M/N)+c_1(M/N)^2\\
		&< 2(c_1(N)+c_1(N_1)).c_1(M/N),
	\end{align*}
	thus $c_1(M/N).(c_1(N)+c_1(N_1))\ge g$. Else $c_1(\det N\otimes N_1)^2\ge 0$, and so $h^0(S,H\otimes(M/N)^\vee)=h^0(S,\det N\otimes N_1)\ge 2$, whereby $M/N\otimes\mathcal{O}_C$ contributes to $\gamma(C)$. Thus we are in case (b).
	
	In case (iii), since $\det E/N \cong \det M/N\otimes N_1$, we have $h^0(S,\det M/N\otimes N_1)\ge 2$ by \cite[Lemma 3.3]{Lelli_Chiesa_2013}. Thus as $h^0(S,\det N)\ge 2$, we have that $\det N\otimes\mathcal{O}_C$ contributes to $\gamma(C)$. Thus proving the first statement of case (c).
	
	In cases (iii) and (iv), as $h^0(S,M/N)<2$ we have $c_1(M/N)^2\le -2$. We now calculate 
	\begin{align*}
		2g-2=c_1(E)^2&= (c_1(N)+c_1(M/N)+c_1(N_1))^2\\
		&= c_1(N)^2+c_1(M/N)^2+c_1(N_1)^2+2c_1(N).c_1(M/N)+2c_1(M/N).c_1(N_1)+2c_1(N).c_1(N_1)\\
		&\le c_1(N)^2+c_1(N_1)^2+2c_1(N).c_1(M/N)+2c_1(M/N).c_1(N_1)+2c_1(N).c_1(N_1)-2\\
		&\le g-1 +c_1(N)^2+2c_1(M/N).c_1(N)-2 ,
	\end{align*} thus $\frac{c_1(N)^2}{2}+c_1(M/N).c_1(N)\ge \frac{g+1}{2}$. If $c_1(N)^2\le-2$, we are in case (d). 
	
	From now on, we assume $c_1(N)^2\ge 0$. From the inequality $\mu(M/N)\ge\mu(N_1)$, we see that $c_1(N).c_1(M/N)>c_1(M/N)^2+c_1(N).c_1(M/N)\ge c_1(N_1)^2+c_1(N_1).c_1(M/N)\ge c_1(N_1).c_1(M/N)$. Thus \[c_1(M/N).c_1(N)-\frac{1}{2}c_1(N).(c_1(M/N)+c_1(N_1)) = \frac{1}{2}(c_1(N).c_1(M/N)-c_1(N).c_1(N_1))>0,\] and we are in case (c).
	
\end{proof}

\begin{lemma}
	With $E$ as above, if general curves in $|H|_s$ have Clifford index $\gamma=\gamma(C)$, we have $d\ge 5+\frac{3}{2}\gamma$.
\end{lemma}
\begin{proof}
	The proof follows the same argument as \Cref{general proof of d ge BLA}.
\end{proof}

\subsection{Filtration \texorpdfstring{$1\subset2\subset3\subset 4$}{}}

We suppose $E$ has a terminal filtration of the form \[0=E_0\subset E_1 \subset E_2 \subset E_3 \subset E_4=E,\] where $\operatorname{rk}(E_i)=i$, and $E_i/E_{i+1}$ are  torsion free sheaves of rank $1$. Furthermore, we have 
\begin{gather}
	\mu(E_1)\ge \mu(E_2)\ge \mu(E_3)\ge \mu(E)=\frac{g-1}{2}\ge \mu(E/E_3)\\
	\mu(E_1)\ge \mu(E_2/E_1)\ge \mu(E_3/E_2)\ge\mu(E/E_3)\\
	\mu(E_i/E_j)\ge \mu(E/E_3) \text{ for $1\le j<i\le 4$}\\
	d=c_1(E/E_3).(c_1(E_1)+c_1(E_2/E_1)+c_1(E_3/E_2))+c_1(E_1).c_1(E_3/E_2)\\+c_1(E_2/E_1).c_1(E_3/E_2)+c_1(E_1).c_1(E_2/E_1)\notag
\end{gather}
Letting $e_i\colonequals c_1(E_i/(E_{i-1}))$, be the Chern roots of $E$, and writing $e_i+e_j\colonequals c_1(E_i/E_{i-1}\otimes E_j/E_{j-1})$, we have 
\begin{align*}
	4d=&e_1(e_2+e_3+e_4)+(e_1+e_2).(e_3+e_4)+(e_1+e_2+e_3).e_4\\&+(e_1+e_4).(e_2+e_3)+(e_1+e_3).(e_2+e_4)+(e_1+e_3+e_4).e_2+(e_1+e_2+e_4).e_3
\end{align*}

\begin{lemma}\label{lemma full flag bound}
	With $E$ as above, if general curves in $|H|_s$ have Clifford index $\gamma=\gamma(C)$, $$m\colonequals \min\{D^2\vert D\in\Pic(S), D^2\ge 0, \text{ $D$ is effective}\}$$ (i.e. there are no curves of genus $g^\prime<\frac{m+2}{2}$ on $S$), and $$\mu=\min\{\mu(D) \vert D\in\Pic(S), D^2\ge0, \mu(D)>0\},$$ we have $d\ge\frac{5}{4}\gamma+\frac{\mu}{2}+\frac{m}{2}+\frac{9}{2}$.
\end{lemma}
\begin{proof}
	From \Cref{LC rank 3 lemma 4.1} and \Cref{Prop quotients contribute to Cliff(C)}, we see that $\det(E/E_i)\otimes\mathcal{O}_C$ contributes to $\gamma(C)$, and so we have $e_1(e_2+e_3+e_4)\ge \gamma+2$, $(e_1+e_2).(e_3+e_4)\ge\gamma+2$, and $(e_1+e_2+e_3).e_4\ge \gamma+2$. We also have $h^2(S,F)=0$ for $F=\det (E_i/E_j)$ and $F=E/E_3,\det E_i$. 
	
	It remains to bound the other four terms. 
	
	To bound $(e_2+e_3).(e_1+e_4)$, we note that $\mu(e_2+ e_3)\ge\mu+\mu(e_3)\ge\mu+\mu(E/E_2)$, and thus $$(e_2+ e_3)^2+(e_1+e_4).(e_2+e_3)\ge\mu+ \frac{(e_1+e_2).(e_3+e_4)}{2}+\frac{(e_3+e_4)^2}{2}\ge\mu+\frac{\gamma+2}{2}+\frac{(e_3+e_4)^2}{2}.$$ Furthermore, we note that $\mu(e_1+e_4)=\mu(e_1)+\mu(e_4)\ge \frac{g-1}{2}+\mu$, whereby $$(e_1+e_4)^2+(e_1+e_4).(e_2+e_3)\ge \gamma.$$ Now if $h^0(S,e_1+e_4)<2$ then by considering the Euler characteristic we have $(e_1+e_4)^2\le-2$, and thus $(e_1+e_4).(e_2+e_3)\ge\gamma+2$. If $h^0(S,e_2+e_3)<2$ then $(e_2+e_3)^2\le-2$, and we have $(e_1+e_4).(e_2+e_3)\ge2+\mu+\frac{\gamma+2}{2}+\frac{(e_3+e_4)^2}{2}$. By assumption, $(e_3+e_4)^2\ge m$, hence $(e_1+e_4).(e_2+e_3)\ge3+\mu+\frac{\gamma}{2}+\frac{m}{2}$ as well. Finally, if $h^0(S,e_1+e_4),h^0(S,e_2+e_3)\ge 2$, and thus they contribute to the $\gamma(C)$, and hence by \Cref{LC rank 3 lemma 4.1} $(e_1+e_4).(e_2+e_3)\ge\gamma+2$. Therefore in either case, we have $(e_1+e_4).(e_2+e_3)\ge3+\mu+\frac{\gamma}{2}+\frac{m}{2}$.
	
	To bound $(e_1+e_3).(e_2+e_4)$, we note that $\mu(e_1+e_3)\ge\frac{g-1}{2}$, and hence $$(e_1+e_3)^2+(e_1+e_3).(e_2+e_4)\ge\frac{g-1}{2}.$$ We also note that $\mu(e_2+e_4)\ge \mu+\mu(E/E_1)\ge\mu+\mu(E/E_2)$, whereby $$(e_2+e_4)^2+(e_1+e_3).(e_2+e_4) \ge 1+ \frac{(e_1+e_2).(e_3+e_4)}{2}+\frac{(e_3+e_4)^2}{2}\ge1+ \frac{\gamma+2}{2}+\frac{(e_3+e_4)^2}{2}.$$ As above, we have $(e_1+e_3).(e_2+e_4)\ge 3+\mu+\frac{\gamma}{2}+\frac{m}{2}$.
	
	To bound $(e_1+e_3+e_4).e_2$, we note that $\mu(e_1+e_3+e_4)\ge\mu(e_1)\ge\frac{g-1}{2}$ and $\mu(e_2)\ge\mu(E/E_1)\ge\mu(E/E_2)$. Following the same argument as above, we have $(e_1+e_3+e_4).e_2\ge 3+\frac{\gamma}{2}+\frac{m}{2}$.
	
	To bound $(e_1+e_2+e_4).e_3$, we note that $\mu(e_1+e_2+e_4)\ge\mu(e_1)\ge\frac{g-1}{2}$ and $\mu(e_3)\ge\mu(E/E_2)$. Following the same argument as above, we have $(e_1+e_2+e_4).e_3\ge 3+\frac{\gamma}{2}+\frac{m}{2}$.
	
	Finally, we have that three of the terms in the expression for $4d$ are bounded below by $\gamma+2$, two by $3+\frac{\gamma}{2}+\frac{m}{2}$, and two by $3+\mu+\frac{\gamma}{2}+\frac{m}{2}$. Thus $d\ge\frac{5}{4}\gamma+\frac{\mu}{2}+\frac{m}{2}+\frac{9}{2}$, as desired.
\end{proof}
\begin{remark}
	We note that in the proof above, $\mu$ is always at least the minimum slope of the determinant of a quotient of $E$. 
\end{remark}

\section{Lifting \texorpdfstring{$\g{3}{d}$}{}s}\label{Section Lifting \texorpdfstring{$\g{3}{d}$}{}s}

As above, $(S,H)$ is a polarized K3 surface of genus $g$, $C\in|H|$ is a smooth irreducible curve of general Clifford index $\gamma=\lfloor\frac{g-1}{2}\rfloor$, $A$ is a complete basepoint free $\g{3}{d}$ with $\rho(A)<0$, and $E=E_{C,A}$ the unstable LM bundle. Having attained the needed bounds on $c_2(E)$, we can prove our lifting results.
\begin{theorem}\label{theorem lifting g3ds general}
	Let $(S,H)$ be a polarized K3 surface of genus $g\neq 2,3,4,8$ and $C\in |H|$ a smooth irreducible curve of Clifford index $\gamma$. Let $$m\colonequals \min\{D^2 ~\vert~ D\in\Pic(S),~ D^2\ge 0, \text{ $D$ is effective}\}$$ (i.e. there are no curves of genus $g^\prime<\frac{m+2}{2}$ on $S$), and $$\mu=\min\{\mu(D) ~\vert~ D\in\Pic(S),~ D^2\ge0,~ \mu(D)>0\}.$$ If $$d<\min\left\{\frac{5}{4}\gamma+\frac{\mu}{2}+\frac{m}{2}+\frac{9}{2},~\frac{5}{4}\gamma+\frac{m}{2}+5,~ \frac{3}{2}\gamma+5,~ \frac{\gamma}{2}+\frac{g-1}{2}+4\right\},$$ then there is a line bundle $M\in\Pic(S)$ adapted to $|H|$ such that $|A|\subseteq |M\otimes\mathcal{O}_C|$ and $\gamma(M\otimes\mathcal{O}_C)\le\gamma(A)$. Moreover, one has $c_1(M).C\le\frac{3g-3}{2}$.
\end{theorem}
\begin{proof}
	The LM bundle $E$ has $c_2(E)=d$. If $g\neq2,3,4,8$, then $d<\frac{5g+19}{6}$. As $$d<\min\left\{\frac{5}{4}\gamma+\frac{\mu}{2}+\frac{m}{2}+\frac{9}{2},~\frac{5}{4}\gamma+\frac{m}{2}+5,~ \frac{3}{2}\gamma+5,~ \frac{\gamma}{2}+\frac{g-1}{2}+4\right\}$$ by assumption, the only terminal filtration of $E$ is of type $1\subset 4$. Thus by \Cref{Prop Proof Strategy}, the result follows. 
\end{proof}

We remark that \Cref{Main Result 1.2} is a special case of \Cref{theorem lifting g3ds general} since if $S$ has no elliptic curves, then $m \geq 2$ so that $d < \frac{5}{4}\gamma + 6$. Considering the bounds obtained in \Cref{Section Filtrations of Lazarsfeld--Mukai Bundles of Rank 4}, we have also proved the following proposition.

\begin{prop}
	With $A$ as above, the bundle $E_{C,A}$ only admits a terminal filtration of type $1\subset 4$, $1\subset 2\subset 4$, or $1 \subset 2 \subset 3 \subset 4$.
\end{prop}
\begin{proof}
	We simply solve $\rho(g,3,d)<0$ for $d$ and compare it to the bounds obtained for each terminal filtration.
\end{proof}

\section{Maximal Brill--Noether Loci in Genus
\texorpdfstring{$\leq 23$}{}}\label{Section Maximal Brill--Noether Loci in Genus \texorpdfstring{$14-23$}{}}

In this section, we identify the maximal Brill--Noether loci in genus $3$--$19$, $22$, and $23$, proving \Cref{Main Result 1.1}. Our technique combines known results about non-containments of Brill--Noether loci, work by Lelli-Chiesa~\cite{Lelli_Chiesa_2013} on lifting of rank $2$ linear systems and linear systems computing the Clifford index, together with our lifting results for rank $3$ linear systems above. 

\subsection{Genus \texorpdfstring{$3$--$6$}{}}\label{Section genus 2-9}

By Clifford's theorem, any Brill--Noether special curve of genus $3$
or $4$ is hyperelliptic, hence $\mathcal{M}^1_{g,2}$ is the only
maximal (and expected maximal) Brill--Noether locus. Similarly, in genus $5$,
every Brill--Noether special curve has gonality $\leq 3$, hence
$\mathcal{M}^1_{5,3}$ is the only maximal (and expected maximal),
Brill--Noether locus.  Thus \Cref{conjecture} holds in genus
$3$--$5$.  In genus 6, we verify the conjecture as well.

\begin{prop}
	The maximal Brill--Noether loci in genus $6$ are $\mathcal{M}^1_{6,3}$ and $\mathcal{M}^2_{6,5}$.
\end{prop}
\begin{proof}
$\mathcal{M}^1_{6,3}$ and $\mathcal{M}^2_{6,5}$ are the expected
maximal Brill--Noether loci. It remains to show that they are
distinct. Since $\rho(6,1,3)=-2$ and $\rho(6,2,5)=-3$, results on the
codimension of Brill--Noether loci (e.g.,
\cite{EdidinThesis,Eisenbud_Harris_1989,Steffen1998}) imply that
$\mathcal{M}^1_{6,3} \nsubseteq\mathcal{M}^2_{6,5}$. A smooth plane
quintic curve $C$ has genus $6$. By a well-known result of Max Noether
\cite{HOTCHKISS_2020}, $C$ has gonality $4$, hence has no
$\g{1}{3}$. Thus $\mathcal{M}^2_{6,5}\nsubseteq \mathcal{M}^1_{6,3}$.
\end{proof}

\subsection{Unexpected containments in genus \texorpdfstring{$7$--$9$}{}}\label{Genus 7-9}

In each genus $7$-$9$, there are two expected maximal Brill--Noether
loci, and we give detailed constructions of the unexpected
containments between them.  In genus $7$ and $9$, we are indebted to
Hannah Larson for pointing them out.  These are $\mathcal{M}^2_{7,6}
\subset \mathcal{M}^1_{7,4}$, $\mathcal{M}^1_{8,4} \subset
\mathcal{M}^2_{8,7}$, and $\mathcal{M}^2_{9,7} \subset
\mathcal{M}^1_{9,5}$.  Thus in these genera, there is a unique maximal
Brill--Noether locus.

\begin{prop}\label{genus 7 exception}
Every Brill--Noether special curve of genus $7$ has a $\g{1}{4}$.
\end{prop}
\begin{proof}
The expected maximal Brill--Noether loci in genus $7$ are
$\mathcal{M}^1_{7,4}$ and $\mathcal{M}^2_{7,6}$. We show that every
smooth genus $7$ curve with a $\g{2}{6}$ has a $\g{1}{4}$. Let
$\phi:C\to\mathbb{P}^2$ be the map given by the $\g{2}{6}$. If the
$\g{2}{6}$ is not very ample, then $C$ has a $\g{1}{4}$. Thus we can
assume $\phi$ is a nondegenerate embedding, so that $\phi(C)$ is a
plane curve of degree $6$, so has arithmetic genus $10$. Hence
$\phi(C)$ must have a singular point of multiplicity $\geq 2$.
Projecting from this point gives a $\g{1}{k}$ for $k \leq 4$, hence a
$\g{1}{4}$.
\end{proof}

\begin{prop}[Mukai~{\cite[Lemma 3.8]{Mukai_Curves_and_grassmannians_1993}}]\label{genus 8 exception}
	Every Brill--Noether special curve of genus $8$ has a $\g{2}{7}$.
\end{prop}
\begin{proof}
The maximal Brill--Noether loci in genus $8$ are $\mathcal{M}^1_{8,4}$
and $\mathcal{M}^2_{8,7}$. We show that a curve $C$ of genus $8$ with
a $\g{1}{4}$ has a $\g{2}{7}$. Let $A$ be a line bundle of type
$\g{1}{4}$ on $C$. If $C$ has a $\g{2}{6}$ then it has a $\g{2}{7}$,
thus we may assume that $C$ has no $\g{2}{6}$, hence no $\g{3}{8}$
(Serre adjoint to a $\g{2}{6}$). Similarly, we can assume $C$ has no
$\g{1}{3}$ (as twice a $\g{1}{3}$ is a $\g{2}{6}$), whence $|A|$ is
basepoint free. Furthermore, the Serre adjoint $A'$ of $A$ is of type
$\g{4}{10}$ and is very ample as there is no $\g{3}{8}$. Hence $|A'|$
exhibits $C$ as degree $10$ curve in $\mathbb{P}^4$. This embedding of
$C$ has 8 trisecant lines by the Berzolari formula
$$\# \{ \text{trisecant lines to $C$}\}=
\frac{(d-2)(d-3)(d-4)}{6}-g(d-4),$$
where $g$ is the genus of $C$ and $d$ is the degree of $C$ in
$\mathbb{P}^4$, see \cite{barz}. Projecting from one of the trisecant
lines gives a $\g{2}{7}$.
\end{proof}

\begin{prop}\label{genus 9 exception}
	Every Brill--Noether special curve of genus $9$ has a $\g{1}{5}$.
\end{prop}
\begin{proof}
The expected maximal Brill--Noether loci in genus $9$ are
$\mathcal{M}^{1}_{ 9,5 }$ and $\mathcal{M}^{2}_{9 ,7 }$.  We will
show, similarly to \Cref{genus 7 exception}, that every smooth genus
$9$ curve with a $\g{2}{7}$ has a $\g{1}{5}$. Let
$\phi:C\to\mathbb{P}^2$ be the map given by the $\g{2}{7}$. If the
$\g{2}{7}$ is not very ample, then $C$ has a $\g{1}{5}$. Thus we can
assume $\phi$ is a nondegenerate embedding, so that $\phi(C)$ is a
plane curve of degree $7$, so has arithmetic genus $15$. Hence
$\phi(C)$ must have a singular point of multiplicity $\geq 2$.
Projecting from this point gives a $\g{1}{k}$ for $k \leq 5$, hence a
$\g{1}{5}$.  
\end{proof}

\begin{remark}
The constructions in genus $7$--$9$ rely on projections from secant
linear spaces.  Given a very ample linear system of type $\g{r}{d}$
defining an embedding $C \to \mathbb{P}^r$ of degree $d$, if $C$
admits a $k$-secant $l$-dimensional linear subspace of $\mathbb{P}^r$,
then projection from that linear subspace results in a
$\g{r-l-1}{d-k}$.  The expected dimension of the space of
$l$-dimensional linear spaces of $\mathbb{P}^r$ that are $k$-secant to
$C$ is classically known to be $k - (k-l-1)(r-l)$, see
\cite{farkas:secant}.  Secant linear spaces for which this expected
dimension is nonnegative (resp.\ negative) are called \emph{expected}
(resp.\ \emph{unexpected}).  When the expected dimension is 0, there
are unwieldy enumerative formulas for the expected number of such
secant linear spaces generalizing the Berzolari formula, see
\cite[VIII.4]{ACGH}.  We have checked that the only cases when an
expected maximal $\g{r}{d}$ (or its Serre adjoint) admits an expected
$k$-secant $l$-dimensional linear space and such that the associated
$\g{r-l-1}{d-k}$ is also Brill--Noether special (and not Serre
adjoint to a $\g{r}{d}$) are the three cases discussed
above in genus $7$--$9$.  Thus no additional unexpected containments
of expected maximal Brill--Noether loci can arise from expected secant
linear spaces.  Unexpected secant linear spaces could potentially give
rise to other unexpected containments, but these should not exist if
we believe various versions of the Donagi--Morrison conjecture for expected maximal
Brill--Noether special linear systems, see
\cite[Theorem~1.4]{Lelli_Chiesa_2015}.
\end{remark}

\subsection{Genus \texorpdfstring{$10$--$13$}{}}\label{Section Lower Genus}

We first establish a few useful lemmas which, in effect, say that if $\Pic(S)=\langle H,L \rangle$ looks like it is obtained by lifting a $\g{r}{d}$ on $C \in |H|$ to a line bundle $L$, then $L$ is in fact a lift of a $\g{r}{d}$. Moreover, for these lifts, we would like the line bundle to be basepoint free, which is true if the $\g{r}{d}$ is primitive. In particular, our next lemma shows that if a curve $C$ on a K3 surface strictly contains a Brill--Noether special linear system, then it is primitive.

\begin{lemma}\label{strict BN special implies primitive}
	Let $(S,H)$ be a polarized K3 surface of genus $g$, $C\in|H|$ a smooth connected curve, and $A\in\Pic(C)$ be a line bundle of type $\g{r}{d}$. Suppose that $\rho(g,r,d)<0$ and $C$ has no Brill--Noether special linear series of Clifford index smaller than $A$. Then $A$ is primitive.
\end{lemma}
\begin{proof}
	We note that $\gamma(\omega_C\otimes A^\vee)=\gamma(A)$, $\rho(A)=\rho(\omega_C\otimes A^\vee)$, $\gamma(A-P)< \gamma(A)$ when $P$ is a basepoint of $A$, and $\rho(g,r,d-1)<\rho(g,r,d)$. Suppose $A$ has a basepoint $P$. Then $A-P$ has strictly smaller Clifford index and is Brill--Noether special. By assumption, $C$ cannot be in the linear series $|A-P|$. Thus $A$ is basepoint free. Likewise, if $\omega_C\otimes A^\vee$ has a basepoint $P$, then $\omega_C\otimes A^\vee -P$ is Brill--Noether special and has smaller Clifford index, which cannot be the case.
\end{proof}

Parts of the following Lemma go back to Farkas in \cite{Farkas_2001} and Rathmann's Theorem (see \cite{Knutsen_2002,Rathmann}).

\begin{lemma}\label{lemma max gon and Cliff dim 1}
	Let $(S,H)$ be a polarized K3 surface of genus $g$ in the Noether--Lefschetz divisor $\mathcal{K}^r_{g,d}$, i.e., with $\Pic(S)$ admitting a primitive embedding of the sublattice \[\Lambda^r_{g,d} = \begin{array}{c|cc}
		\multicolumn{1}{c}{} & H & L \\\cline{2-3}
		H & 2g-2 &d \\
		L &d & 2r-2
	\end{array}\] 
	Let $C\in |H|$ be a smooth irreducible curve. 
	\begin{enumerate}[label=(\roman*)]
		\item If $\Pic(S)=\Lambda^r_{g,d}$ and $2\le r,d\le g-1$, then $L$ is nef. 
		\item If $L$ and $H-L$ are basepoint free, $r\ge 2$, and $0< d \le g-1$, then $L\otimes\mathcal{O}_C$ is a $\g{r}{d}$. (The assumption on basepoint free-ness is achieved if for example $S$ has no $(-2)$-curves, or can be checked numerically.)
		\item Suppose that $L\otimes\mathcal{O}_C$ is a $\g{r}{d}$ with $\gamma(r,d)>\lfloor \frac{g-1}{2}\rfloor$ and $\rho(g,r,d)<0$ and that all lattices obtained by lifting special linear systems of general Clifford index or lower cannot be contained in $\Pic(S)$. Then $C$ has Clifford index $\gamma(C)=\lfloor\frac{g-1}{2}\rfloor$, maximal gonality $\lfloor\frac{g+3}{2}\rfloor$, and Clifford dimension $1$. 
		\item If $\Pic(S)=\Lambda^r_{g,d}$ is associated to an expected maximal $\g{r}{d}$, then the assumption on lattices in (iii) holds.
		\item Suppose that $\gamma(r,d)\le \lfloor \frac{g-1}{2}\rfloor$, $\rho(g,r,d)<0$, and that all lattices obtained by lifting special linear systems $A$ not of type $\g{r}{d}$ with $\gamma(A)\le\lfloor\frac{g-1}{2}\rfloor$ cannot be contained in $\Pic(S)$. Then $L\otimes\mathcal{O}_C$ is a $\g{r}{d}$ and $\gamma(C)=\gamma(r,d)$.
	\end{enumerate}   
\end{lemma}
\begin{proof}
	To prove $(i)$ we show that for any $(-2)$-curve $\Gamma=aH+bL\in\Lambda^r_{g,d}$, we have $\Gamma.L\ge 0$. We note that as $\Gamma$ is a $(-2)$-curve, $a$ and $b$ must have opposite sign. We prove (i) in three cases. 
	
	First suppose $a>0$ and $b<0$. Then as $\Gamma.H\ge 1$ and $a>0$, we have $b\Gamma.L\le -2$, thus as $b<0$, $\Gamma.L\ge 0$.

	Second, suppose $a<-1$ and $b>0$. Then since $\Gamma.H\ge 1$, we have $a\Gamma.H\le -2$. Thus $b.\Gamma.L\ge 0$, and since $b>0$ we must have $\Gamma.L\ge 0$.
	
	Lastly, suppose $a=-1$ and $b>0$. We see that if $\Gamma.H\ge2$, then we can follow the same argument as above to see that $L$ is nef. Thus the only remaining case is when $a=-1$ and $\Gamma.H=1$. We calculate $2g-2=(H+\Gamma)^2=(bL)^2=b^2(2r-2)$, hence $b^2=\frac{g-1}{r-1}\in\mathbb{Z}$. From $\Gamma.H=1$, we see $b=\frac{2g-1}{d}$, and plugging this in to $2g-2=b^2(2r-2)$ yields $$d^2(g-1)=(2g-1)^2(r-1).$$ Looking modulo  $g-1$, we immediately see that $r-1\equiv 0\mod g-1$, hence $\frac{r-1}{g-1}\in\mathbb{Z}$, and thus $r=g$, which is a contradiction. Thus $L$ is always nef. 
	
	To prove $(ii)$, we note that $L$ is clearly a lift of a $\g{r^\prime}{d}$ on $C$ for some $r^\prime\ge 0$. Since $0<d\le g-1$, we see that $L^2, (H-L)^2>0$. Furthermore, since $H.L,H.(H-L)>0$, both these line bundles are nontrivial and intersect $H$ positively, hence $h^0(S,L),h^0(S,H-L)\ge 2$. By assumption, $L$ and $H-L$ are basepoint free, and thus globally generated. Therefore \Cref{LC implies 2r-2} applies. Thus, as $L^2=2r-2$, we see that $L\otimes\mathcal{O}_C$ must be a divisor of type $\g{r}{d}$. Hence $(ii)$ is proved. 
	
	To prove $(iii)$, we note that a $\g{1}{d^\prime}$ with $\rho(g,1,d^\prime)<0$ has Clifford index $\gamma(\g{1}{d^\prime})<\lfloor\frac{g-1}{2}\rfloor$. Suppose for contradiction that $C$ has lower than general Clifford index. Then by \cite[Theorem 4.2]{Lelli_Chiesa_2015} we would be able to lift some special linear system computing $\gamma(C)$ to a divisor $L^\prime\in\Pic(S)$, and by assumption $\langle H,L^\prime\rangle$ cannot be contained in $\Pic(S)$. Thus $C$ has general Clifford index. The same argument shows that $C$ cannot have a special linear system computing its Clifford index. Thus $C$ has a $\g{1}{\lfloor\frac{g+3}{2}\rfloor}$ which computes the Clifford index. Hence $C$ has maximal gonality and Clifford dimension $1$. 
	
	To prove $(iv)$, we note that if $C$ had any Brill--Noether special $\g{r^\prime}{d^\prime}$ with $\gamma(\g{r^\prime}{d^\prime})\le\frac{g-1}{2}$, then it has a $\g{r}{d}$ with $\gamma=\frac{g-1}{2}$ or a $\g{1}{d}$ with $\gamma(\g{1}{d})=\frac{g-1}{2}-1$. Thus we only need to consider lattices $\Lambda^r_{g,d}$ associated to those $\g{r}{d}$. The proof is now \Cref{prop lattices are not conatined in exp max}$(i)$. Thus $(iv)$ is proved. 
	
	To prove $(v)$, we note again that $L\otimes\mathcal{O}_C$ is a $\g{r^\prime}{d}$. If $r'\neq r$, then $\gamma(C)\neq\gamma(r,d)$ and some line bundle $A$ would compute $\gamma(C)$. Thus there would exists some lift of $A$ to a line bundle $L^\prime$, but again the lattice $\langle H,L^\prime\rangle\nsubseteq \Pic(S)$. Hence $r^\prime=r$ and we see that $L\otimes\mathcal{O}_C$ is a $\g{r}{d}$. Similarly, $\gamma(C)=\gamma(r,d)$.
\end{proof}
\begin{remark}\label{Checking if L|C is a grd}
	If $\Lambda^r_{g,d}$ has a $(-2)$-curve, there are still some ways to check that $L$ and $H-L$ are basepoint free. Namely, if they are both nef, then we can check they are basepoint free by checking if there are any elliptic curves on $S$. Namely if $N\in\Pic(S)$ is nef and there are no elliptic curves, then $N$ is basepoint free by a well-known result of Saint-Donat. To numerically check if $D\in\Pic(S)$ is nef, one can check whether $D.\Gamma\ge 0$ for any $(-2)$-curve $\Gamma$.
	
	One can also check that $L\otimes\mathcal{O}_C$ is a $\g{r}{d}$ by enumerating all of the degree $d$ $\g{r^\prime}{d}$ on $C$ and using Lelli-Chiesa's lifting results to show that $\Pic(S)$ cannot have a lift of a $\g{r^\prime}{d}$ for $r^\prime\neq r$.
\end{remark}
We can now prove that the maximal Brill--Noether loci in genus
$10$--$19$, $22$, and $23$ are as predicted by
Conjecture~\ref{conjecture}. The proof in genus $10$--$13$ uses
Brill--Noether theory for curves of fixed gonality and various results
distinguishing lattices above. In genus $14$--$19$, $22$ and $23$, the
main strategy to distinguish the expected maximal Brill--Noether loci,
for example to show that $\mathcal{M}^r_{g,d}\nsubseteq
\mathcal{M}^{r^\prime}_{g,d^\prime}$, is to prove that for a very
general K3 surface $(S,H)\in\mathcal{K}^r_{g,d}$, a curve $C\in|H|$
has a $\g{r}{d}$ but not a $\g{r^\prime}{d^\prime}$. This is done by,
first, applying Lemma~\ref{lemma max gon and Cliff dim 1} to deduce
that $C$ has a $\g{r}{d}$, and second, assuming that $C$ has a
$\g{r^\prime}{d^\prime}$ and then using various lifting results to
produce a line bundle $M$ on $S$ that is numerically incompatible with
$\Pic(S)$. 

For the rest of the section, we summarize the various arguments,
organized by genus.

In low genus, where there are no non-computing linear systems, we
argue by the Clifford index of $C$ and can assume that a $\g{r}{d}$
computes the Clifford index of $C\in|H|$. Then Lelli-Chiesa's
lifting results \cite{Lelli_Chiesa_2015} suffice to verify \Cref{conjecture} in genus $10$--$13$.
\begin{prop}\label{theorem low genus in paper}
	For any $10\le g\le13$ and any positive integers $r,d,r',d'$ such that 
	\begin{itemize}
		\item $r^\prime\ge 2$,
		\item $\rho(g,r,d),\rho(g,r',d')<0,$
		\item $\Delta(g,r,d),\Delta(g,r',d')<0,$ and
		\item $2<\gamma(r',d')\leq\gamma(r,d)\le \lfloor\frac{g-1}{2}\rfloor$,
	\end{itemize} there is a polarized K3 surface $(S,H)\in\mathcal{K}^r_{g,d}$ such that a curve $C\in |H|$ admits a $\g{r}{d}$ but not a $\g{r'}{d'}$. Thus $\mathcal{M}^r_{g,d}\nsubseteq\mathcal{M}^{r'}_{g,d'}$.
\end{prop}
\begin{proof}
	First assume that $r^\prime \ge 2$. We let $(S,H)\in\mathcal{K}^r_{g,d}$ be a very general and $C\in|H|$ a smooth irreducible curve of genus~$g$. As in \Cref{prop lattices are not conatined in exp max} ($i$), no lattices obtained by lifting special linear systems on $C$ can be contained in $\Pic(S)$. By \Cref{lemma max gon and Cliff dim 1} ($v$) we see that $L\otimes\mathcal{O}_C$ is a $\g{r}{d}$ and $\gamma(C)=\gamma(r,d)$. We suppose for contradiction that $C$ admits a $\g{r^\prime}{d^\prime}$. We cannot have $\gamma(r^\prime,d^\prime)<\gamma(r,d)$, as then the $\g{r}{d}$ does not compute the Clifford index of $C$. Hence $\gamma(r^\prime,d^\prime)=\gamma(r,d)$. But now \cite[Theorem 4.2]{Lelli_Chiesa_2015} shows that we have a Donagi--Morrison lift $M\in\Pic(S)$ of the $\g{r^\prime}{d^\prime}$, and by \Cref{prop lattices are not conatined in exp max} ($i$) again, we see that $\langle H,M\rangle \nsubseteq \Pic(S)$ unless the Donagi--Morrison lift of the $\g{r^\prime}{d^\prime}$ is of type $\g{r}{d}$, which only occurs when $r, r^\prime\ge 2$. In this case, the Lazarsfeld--Mukai bundle $E_{C,\g{r^\prime}{d^\prime}}$ has a quotient $E$ with $\gamma(E)=0$, and one checks that none of the cases of \Cref{LC Cor 2.5} can occur (for a detailed computation see \Cref{genus 16 theorem}). Thus $C$ cannot admit a $\g{r^\prime}{d^\prime}$.	
\end{proof}
\begin{remark}
	When $r^\prime=1$, case (a) of \Cref{LC Cor 2.5} can occur, hence we assume $r^\prime\ge 2$. In fact, in genus $11$, $\mathcal{M}^1_{11,5}\subseteq \mathcal{M}^2_{11,9}$. However, for expected maximal loci, the codimensions of the expected maximal $\mathcal{M}^1_{g,k}$ and $\mathcal{M}^2_{g,d}$ loci rule out similar containments. 
\end{remark}

\begin{cor}
	In genus $10$--$13$, \Cref{conjecture} holds.
	The maximal Brill--Noether loci
	\begin{itemize}
		\item in genus $10$ are $\mathcal{M}^{1}_{ 10,5 }$ and $\mathcal{M}^{2}_{10 ,8 }$;
		\item in genus $11$ are $\mathcal{M}^{1}_{ 11, 6}$ and $\mathcal{M}^{2}_{ 11, 9}$;
		\item in genus $12$ are $\mathcal{M}^{1}_{ 12,6 }$, $\mathcal{M}^{2}_{ 12,9 }$, and $\mathcal{M}^{3}_{12 ,11 }$;
		\item in genus $13$ are $\mathcal{M}^{1}_{ 13, 7}$, $\mathcal{M}^{2}_{ 13,10 }$, and $\mathcal{M}^{3}_{13 ,12 }$.
	\end{itemize}
\end{cor}
\begin{proof}
	Propositions \ref{theorem low genus in paper} and \ref{Prop submax gonality not contained in noncomputing locus} suffice to verify the conjecture in genus $10$--$13$.
\end{proof}

\subsection{Genus \texorpdfstring{$14$--$15$}{}}

The arguments in genus $14$ and $15$ only require the lifting results
of Lelli-Chiesa~\cite{Lelli_Chiesa_2015} and the preliminary results above.
\begin{prop}
	In genus $14$, the maximal Brill--Noether loci are $\mathcal{M}^{1}_{14,7}$, $\mathcal{M}^{2}_{14,11}$, and $\mathcal{M}^{3}_{14,13}$.
\end{prop}
\begin{proof}
	These loci are the expected maximal Brill--Noether loci in
        genus $14$, thus it remains to show that there are no
        containments among them. By \Cref{Prop submax gonality not
        contained in noncomputing locus},
        $\mathcal{M}^1_{14,7}\nsubseteq \mathcal{M}^2_{14,11}$ and
        $\mathcal{M}^1_{14,7}\nsubseteq \mathcal{M}^3_{14,13}$. By
        \Cref{lemma max gon and Cliff dim 1} $(iii)$, we see that
        there are curves which admit a $\g{2}{11}$ or a $\g{3}{13}$
        and have maximal gonality $\lfloor\frac{14+3}{2}\rfloor=8$,
        whereby $\mathcal{M}^2_{14,11}\nsubseteq\mathcal{M}^1_{14,7}$
        and
        $\mathcal{M}^3_{14,13}\nsubseteq\mathcal{M}^1_{14,7}$. Since
        $\rho(14,2,11)=-1$ and $\rho(14,3,13)=-2$, and noting that
        therefore $\mathcal{M}^2_{14,11}$ has codimension $1$ and
        $\mathcal{M}^3_{14,13}$ has codimension at least $2$ in
        $\mathcal{M}_{14}$, we see that
        $\mathcal{M}^2_{14,11}\nsubseteq\mathcal{M}^3_{14,13}$. Finally,
        Lelli-Chiesa's lifting of rank $2$ linear systems \cite{Lelli_Chiesa_2013} shows that $\mathcal{M}^3_{14,13}\nsubseteq \mathcal{M}^2_{14,11}$.
\end{proof}

The proof in genus $15$ follows the same argument as genus $14$ above.

\begin{prop}
	In genus $15$, the maximal Brill--Noether loci are $\mathcal{M}^{1}_{15,7}$, $\mathcal{M}^{2}_{15,11}$, and $\mathcal{M}^{3}_{15,14}$.
\end{prop}

\subsection{Genus \texorpdfstring{$16$--$17$}{}}

In genus $16$ and $17$, the proofs are slightly complicated by the
fact that one cannot expect to always lift a linear system
$A\in\Pic(C)$ to a line bundle on $S$, but under the Donagi--Morrison
conjecture, we can at least find a Donagi--Morrison lift, i.e., a line bundle
$N\in\Pic(S)$ such that $|A|\subseteq |N\otimes\mathcal{O}_C|$ with
$\gamma(N\otimes\mathcal{O}_C)\le\gamma(A)$, see \Cref{defn:DM_lift}.

\begin{prop}\label{genus 16 theorem}
	The maximal Brill--Noether loci in genus $16$ are $\mathcal{M}^{1}_{16,8}$, $\mathcal{M}^{2}_{16,12}$, and $\mathcal{M}^{3}_{16,14}$.
\end{prop}
\begin{proof}
	As above, it remains to show that there are no containments among these loci.	
	One can check, as in \Cref{Checking if L|C is a grd}, that for $L$ in $\Lambda^3_{16,14}$, $L\otimes\mathcal{O}_C$ is in fact a $\g{3}{14}$. We note that there are no $(-2)$-curves in $\Lambda^2_{15,12}$. Hence \Cref{lemma max gon and Cliff dim 1} applies for $\Pic(S)$ either $\Lambda^2_{16,12}$ or $\Lambda^3_{16,14}$. Thus $\mathcal{M}^2_{16,12}\nsubseteq \mathcal{M}^1_{16,8}$ and $\mathcal{M}^3_{16,14}\nsubseteq \mathcal{M}^1_{16,8}$. Furthermore, we have $\mathcal{M}^1_{16,8}\nsubseteq\mathcal{M}^2_{16,12}$ and $\mathcal{M}^1_{16,8}\nsubseteq\mathcal{M}^3_{16,14}$ from \Cref{Prop submax gonality not contained in noncomputing locus}. Since $\rho(16,2,12)=-2$ and $\rho(16,3,14)=-4$, we see that $\mathcal{M}^2_{16,12}\nsubseteq \mathcal{M}^3_{16,14}$. It remains to show that there are curves with a $\g{3}{14}$ and no $\g{2}{12}$. 
	
	Suppose that $\Pic(S)=\Lambda^3_{16,14}$, and suppose $C$ has
        a line bundle $A$ of type $\g{2}{12}$. Then by \cite[Theorem
        1]{Lelli_Chiesa_2013}, there is a Donagi--Morrison lift of
        $A$. It can easily be checked that if the Donagi--Morrison
        lift $M$ is not of type $\g{3}{14}$, then $M$ can not be
        contained in $\Pic(S)$. Thus we can assume that $M$ is of type
        $\g{3}{14}$ and $M^2=4$. However, by \Cref{coker is gLM}, we
        see that $\gamma(E_{C,A}/N)=0$, and each of the cases in
        \Cref{LC Cor 2.5} cannot hold. In case (c), one appeals to
        \cite[Theorem 5.2]{Saint_Donat_Proj_Models_of_K3s} which shows
        that a curve is hyperelliptic only if there is an irreducible
        curve $B\subset S$ of genus~$1$ or~$2$. However, this would
        yield $B^2=0$ or $B^2=2$, both of which are too small. Thus
        there can be no such~$M$, and thus $C$ cannot admit a $\g{2}{12}$. Thus  $\mathcal{M}^3_{16,14}\nsubseteq\mathcal{M}^2_{16,12}$.
\end{proof}

The proof in genus $17$ follows the same argument as genus $16$ above.

\begin{prop}
	The maximal Brill--Noether loci in genus $17$ are $\mathcal{M}^{1}_{17,9}$, $\mathcal{M}^{2}_{17,13}$, and $\mathcal{M}^{3}_{17,15}$.
\end{prop}

\subsection{Genus \texorpdfstring{$18$}{}}\label{subsec:g18}

The proof in genus $18$ is slightly complicated by the fact that in showing the non-containment $\mathcal{M}^2_{18,13}\nsubseteq \mathcal{M}^3_{18,16}$, the bound in \Cref{theorem lifting g3ds general} does not rule out the possibility of a $1\subset 2\subset 4$ terminal filtration. The other non-contaimnents are similar to the proofs above. We give a proof of this non-trivial non-containment.

\begin{prop}
	The maximal Brill--Noether loci in genus $18$ are $\mathcal{M}^{1}_{18,9}$, $\mathcal{M}^{2}_{18,13}$, and $\mathcal{M}^{3}_{18,16}$.
\end{prop}
\begin{proof}
	The only non-containment requiring additional analysis is $\mathcal{M}^2_{18,13}\nsubseteq \mathcal{M}^3_{18,16}$. The other non-containments follow the arguments above. 
	
	In \Cref{theorem lifting g3ds general}, the bound on $d$ to ensure that a Donagi--Morrison lift exists for a $\g{3}{16}$ on a general $(S,H)\in\mathcal{K}^2_{18,13}$ is $16$, and hence we are not guaranteed to have a Donagi--Morrison lift by using \Cref{Prop Proof Strategy}. However, the bound is sufficient to show that the LM bundle $E_{C,A}$ associated to the $\g{3}{16}$ can only have a terminal filtration of type $1\subset 4$ or $1\subset 2\subset 4$. We argue that the terminal filtration of type $1\subset 2 \subset 4$ cannot exist. 
	
	Suppose $\Pic(S)=\Lambda^2_{18,13}$, and that $C$ has $\g{3}{16}$. \Cref{lemma max gon and Cliff dim 1} shows that $C$ has $\gamma(C)=8$. Suppose also that $E=E_{C,\g{3}{16}}$ has a $1\subset 2\subset 4$ terminal filtration, which is $0\subset N\subset M\subset E$ where $N$ is a line bundle and $M$ has rank $2$. We show that this leads to a contradiction. We have $c_1(N).c_1(E/N)\ge \gamma(C)+2$ by \Cref{Prop quotients contribute to Cliff(C)}. Furthermore, $C$ has general Clifford index by \Cref{lemma max gon and Cliff dim 1}. Up to replacing $N$ with its saturation, we can assume $E/N$ is a gLM bundle of type (II), and a computation gives $\gamma(E)=c_1(N).c_1(E/N)+\gamma(E/N)-2$, thus $\gamma(E/N)\le 2$. 
	
	One can easily check that $S$ has no elliptic curves, hence
        one of the four cases in \Cref{Prop Clifford Index 1 and 2
        gLM} occur. In case $(i)$ and $(ii)$, one checks the cases in
        \Cref{LC Cor 2.5}, and finds that none can occur. Thus for a
        smooth irreducible $D \in |\det(E/N)|$, $D$ is either
        trigonal, a plane quintic, or a plane sextic, see
        \Cref{rem:bounding_invariants}. If $C$ is hyperelliptic or
        trigonal, one finds a Donagi--Morrison lift of the $\g{1}{2}$
        or the $\g{1}{3}$, which cannot be contained in
        $\Pic(S)$. Thus we may assume $\gamma(D)=2$.  As the
        condition $(\ast)$ from \cite[Theorem 4.2]{Lelli_Chiesa_2015}
        applies, we obtain a Donagi--Morrison lift of the $\g{2}{6}$,
        which again cannot be contained in $\Pic(S)$. Thus $E$ cannot
        have a $1\subset 2\subset 4$ filtration.
	
	Therefore $E$ can only have a terminal filtration of type
        $1\subset 4$, and \Cref{conj DM} holds for the
        $\g{3}{16}$. The rest of the argument is now similar to the arguments above.
\end{proof}

\subsection{Genus \texorpdfstring{$19$}{}}\label{subsection genus 19}

\begin{prop}
\label{prop:g19}
		The maximal Brill--Noether loci in genus $19$ are $\mathcal{M}^{1}_{19,10}$, $\mathcal{M}^{2}_{19,14}$, and $\mathcal{M}^{3}_{19,17}$.
\end{prop}
\begin{proof}
	To apply \Cref{theorem lifting g3ds general}, it suffices to
        note that when $\Pic(S)=\Lambda^2_{19,14}$, then we have
        $\mu\ge 2$ and hence the Donagi--Morrison conjecture holds for
        a $\g{3}{17}$ on a smooth $C \in |H|$, otherwise the argument is similar to \Cref{genus 16 theorem}.
\end{proof}

\begin{remark}
\label{remark:counterexample}
In \cite[Appendix A, Remark 12]{Lelli_Chiesa_2015}, Knutsen and
Lelli-Chiesa construct examples of K3 surfaces $S$ of Picard rank 2
such that a smooth irreducible curve $C \subset S$ has a
Brill--Noether special linear system $A$ of rank $3$ with $\rho(A)=-1$
whose Lazarseld--Mukai bundle $E_{C,A}$ admits no effective sub-line
bundle. That is, \Cref{Prop Proof Strategy} cannot be used to find a
Donagi--Morrison lift of $A$. Here, we give an explicit example and explain how it
relates to our results.

We first recall Knutsen and Lelli-Chiesa's construction. For even
integers $a, b \geq 4$ and $d=a+b$, let $S$ be a
K3 surface with $\Pic(S)=\Lambda^b_{a,d}$.
Suppose that $\Pic(S)$ has no classes of self-intersection $-2$ or
$0$. There are infinitely many choices of $a$ and $b$ that satisfy
these hypotheses, and such that every element of the linear systems
$|H|$ and $|L|$ are reduced and irreducible; these are examples of the
so-called \emph{Knutsen K3 surfaces} in \cite{arap_marshburn}. Thus
general curves $C_1 \in |H|$ and $C_2 \in |L|$ are smooth of genus $a$
and $b$, and by Lazarsfeld's theorem~\cite{lazarsfeld1986}, are
Brill--Noether general, in particular, have generic gonality
$k_1=(a+2)/2$ and $k_2=(b+2)/2$, respectively. Let $E_1$ and $E_2$ be
the LM bundles associated to gonality pencils $\g{1}{k_1}$ on $C_1$
and a $\g{1}{k_2}$ on $C_2$. As these pencils are Brill--Noether
general, the LM bundles $E_1$ and $E_2$ are simple, hence admit no
injective map from an effective line bundle $N$. A calculation using
\Cref{remark glm is lm bundle} shows that the vector bundle
$E=E_1\oplus E_2$ is a LM bundle associated to a linear system $A$ of
type $\g{3}{k_1+k_2+d}$ on a smooth irreducible curve $C \in
|H+L|$. We note that $g(C)=2d-1$, and that $\rho(A)=-1$. However,
since $E$ admits no injective map $N\hookrightarrow E$, the linear
system $A$ admits no Donagi--Morrison lift, and so \Cref{conj DM}
fails for $(C,A)$.

By construction, $E$ has a $2\subset 4$ terminal filtration. Checking
the bound from \Cref{lemma bound 2<4 filt}, one finds that
$\gamma(C)\le d-2$, thus $C$ does not have general Clifford index. In
fact, one can verify using \Cref{lemma max gon and Cliff dim 1} that $L\vert_C$ is a line bundle of type
$\g{b}{d+2b-2}$, which has $\gamma(L\vert_C)=d-2$. We note that $A$ is
non-computing, and does not compute the special Clifford index
$\widetilde{\gamma}(C)$. However, the linear system $L\vert_C$ does
compute $\widetilde{\gamma}(C)$, and has a (Donagi--Morrison) lift by construction.
Hence, while this is a counterexample to the Donagi--Morrison
conjecture, it does not give a negative answer to \Cref{question2}.

The first case where such an example shows the failure of \Cref{conj
DM} for $(C,A)$ is genus $19$, with $a=6$ and $b=4$.  The
corresponding polarized K3 surface $(S,H+L)$ of genus 19 has
$\Pic(S)=\Lambda^{4}_{19,16}$ with basis $H+L,L$. In the proof of
\Cref{prop:g19}, we needed the Donagi--Morrison Conjecture (\Cref{conj
DM}) for linear systems on curves on a different lattice polarized K3
surface, showing that our bounded version (\Cref{theorem lifting g3ds
general}) is in some sense tight (at least in genus 19).
\end{remark}

\subsection{Genus \texorpdfstring{$20-21$}{}}
We briefly list what is known and summarize the last needed non-containments to verify \Cref{conjecture} in genus $20$ and $21$.

The expected maximal Brill--Noether loci in genus $20$ are $\mathcal{M}^{1}_{20,10}$, $\mathcal{M}^{2}_{20,15}$, and $\mathcal{M}^{3}_{20,17}$, and $\mathcal{M}^{4}_{20,19}$. We state the following propositions without proof, as they follow the arguments above.

\begin{prop}
	In genus $20$, the loci $\mathcal{M}^1_{20,10}$, $\mathcal{M}^2_{20,17}$, and $\mathcal{M}^4_{20,19}$ are maximal. There are also non-containments 
	\begin{itemize}
		\item $\mathcal{M}^3_{20,17}\nsubseteq\mathcal{M}^1_{20,10}$ and
		\item $\mathcal{M}^3_{20,17}\nsubseteq\mathcal{M}^2_{20,17}$.
	\end{itemize}
\end{prop}

In fact, the only non-containment that remains to verify \Cref{conjecture} in genus $20$ is $\mathcal{M}^3_{20,17}\nsubseteq\mathcal{M}^4_{20,19}$. Current lifting methods do not suffice to prove the last non-containment, as there are no known general lifting results for linear systems of rank~$4$. If \Cref{conj DM} holds in rank $4$, then this would suffice. Another approach to verifying \Cref{conjecture} in genus $20$ is to show that the codimension of $\mathcal{M}^3_{20,17}$ is at least the expected value of $4$ and the codimension of $\mathcal{M}^4_{20,19}$ is at least the expected value of $5$.

Similarly, the expected maximal Brill--Noether loci in genus $21$ are $\mathcal{M}^{1}_{21,11}$, $\mathcal{M}^{2}_{21,15}$, and $\mathcal{M}^{3}_{21,18}$, and $\mathcal{M}^{4}_{21,20}$. And current methods suffice to prove that some expected maximal loci are indeed maximal.

\begin{prop}
	In genus $21$, the loci $\mathcal{M}^1_{21,11}$ and $\mathcal{M}^4_{21,20}$ are maximal. There are also non-containments
	\begin{itemize}
		\item $\mathcal{M}^2_{21,15}\nsubseteq\mathcal{M}^1_{21,11}$
		\item $\mathcal{M}^3_{21,18}\nsubseteq\mathcal{M}^1_{21,11}$
	\end{itemize}
\end{prop}

Our results reduce the verification of \Cref{conjecture} in genus $21$ to the verification of the non-containments $\mathcal{M}^2_{21,15}\nsubseteq\mathcal{M}^4_{21,20}$ and $\mathcal{M}^3_{21,18}\nsubseteq\mathcal{M}^4_{21,20}$. Again, \Cref{conj DM} in rank $4$ would suffice. Another approach is by verifying that the codimension of $\mathcal{M}^4_{21,20}$ is the expected value of $4$, since $\rho(21,2,15)=\rho(21,3,18)=-3$ and thus the corresponding loci have codimension $3$ in $\mathcal{M}_{21}$.
	
\subsection{Genus \texorpdfstring{$22$}{}}

\begin{prop}
	The maximal Brill--Noether loci in genus $22$ are $\mathcal{M}^{1}_{22,11}$, $\mathcal{M}^{2}_{22,16}$, and $\mathcal{M}^{3}_{22,19}$, and $\mathcal{M}^{4}_{22,21}$. 
\end{prop}
\begin{proof}
	 In genus $22$,	\cite[Corollary 3.5]{CHOI2022} shows that the loci $\mathcal{M}^2_{22,16}$ and $\mathcal{M}^3_{22,19}$ are distinct. The argument then follows \Cref{genus 16 theorem}.
\end{proof}

\subsection{Genus \texorpdfstring{$23$}{}}

Finally, we provide a proof in genus $23$. We note that Farkas proved in \cite{Farkas2000} that the Brill--Noether divisors $\mathcal{M}^1_{23,12}$, $\mathcal{M}^2_{23,17}$, and $\mathcal{M}^3_{23,20}$ are mutually distinct. Our results, and those of Lelli-Chiesa \cite{Lelli_Chiesa_2013}, provide a different proof for these non-containments.  However, the full proof of \Cref{conjecture} in genus $23$ requires our improved lifting results.

\begin{prop}\label{genus_23}
	The maximal Brill--Noether loci in genus $23$ are $\mathcal{M}^{1}_{23,12}$, $\mathcal{M}^{2}_{23, 17}$, $\mathcal{M}^{3}_{23,20 }$, and $\mathcal{M}^{4}_{23,22 }$.
\end{prop}

\begin{proof}
	Since $\rho(23,1,12)=\rho(23,2,17)=\rho(23,3,20)=-1$ and $\rho(23,4,22)=-2$, Eisenbud and Harris \cite{Eisenbud_Harris_1989} show that the corresponding loci are irreducible of codimension $1$ in $\mathcal{M}_g$ and that $\mathcal{M}^4_{23,22}$ has codimension~$\ge2$, hence the other loci cannot be contained in $\mathcal{M}^4_{23,22}$. Since there are no $(-2)$-curves in the Picard lattices of a general K3 surface in $\mathcal{K}^{2}_{23, 17}$, $\mathcal{K}^{3}_{23,20}$, and $\mathcal{K}^{4}_{23,22}$, we see by \Cref{lemma max gon and Cliff dim 1} that none of the loci are contained in $\mathcal{M}^1_{23,12}$. One can check that for a very general K3 surface in $\mathcal{K}^4_{23,22}$, the minimal positive self-intersection is $4$. Hence by \Cref{theorem lifting g3ds general}, if $C\in|H|$ had a $\g{3}{20}$ then by considering the Donagi--Morrison lifts, one finds that $L$ is the only possible Donagi--Morrison lift of the $\g{3}{20}$. Therefore $\gamma(E/N)=0$, and one then argues as in the proof of \Cref{genus 16 theorem}. Thus $\mathcal{M}^4_{23,22}\nsubseteq\mathcal{M}^3_{23,20}$. The lifting results in \cite{Lelli_Chiesa_2013} similarly show that $\mathcal{M}^4_{23,22}\nsubseteq\mathcal{M}^2_{23,17}$ and $\mathcal{M}^3_{23,22}\nsubseteq\mathcal{M}^2_{23,17}$. Since the latter two are codimension $1$ and irreducible, they are distinct. Thus all of the Brill--Noether loci are distinct. 
\end{proof}

\vfill
\end{document}